\documentclass[11pt,oneside,reqno]{amsart}

\usepackage{amsmath,amssymb,amsthm,textcomp,mathtools}
\usepackage{amsfonts,graphicx}
\usepackage[mathscr]{eucal}
\usepackage{color}
\usepackage{csquotes}
\usepackage{enumitem}
\numberwithin{equation}{section}

\theoremstyle{definition}

\numberwithin{equation}{section}

\newtheorem{theorem}{\bf Theorem}[section]
\newtheorem{remark}{\bf Remark}[section]
\newtheorem{proposition}{Proposition}[section]

\newtheorem{definition}{Definition}[section]
\newtheoremstyle
{remarkstyle}
{}
{11pt}
{}
{}
{\bfseries}
{:}
{     }
{\thmname{#1} \thmnumber{#2} }

\theoremstyle{remarkstyle}

\begin{document}
	\title{Fractional Compound Poisson Random Fields on Plane}
	 \author[P. Vishwakarma]{Pradeep Vishwakarma}
	 \address{P. Vishwakarma, Department of Mathematics,
	 	Indian Institute of Technology Bhilai, Durg-491002, India.}
	 \email{pradeepv@iitbhilai.ac.in}
\author[K. K. Kataria]{Kuldeep Kumar Kataria}
	 \address{K. K. Kataria, Department of Mathematics,
		 	Indian Institute of Technology Bhilai, Durg-491002, India.}
	 \email{kuldeepk@iitbhilai.ac.in}

	
	\date{\today}
	
	\maketitle
	\begin{abstract}
		We study a compound Poisson random field on plane and examine its various fractional variants. We derive the distributions of these random fields and in some particular cases, obtain their associated system of governing differential equations. Additionally, various weak convergence results are derived. Also, we introduce and study some fractional variants of the Poisson random field (PRF) that includes the space fractional and the space-time fractional Poisson random fields. We discuss various distributional properties for the fractional PRFs and obtain their time-changed representations. Lastly, we define and study some  fractional compound Poisson random fields using the fractional variants of PRF.
	\end{abstract}

\section{Introduction}
The compound Poisson process generalizes the standard Poisson process by permitting random jump sizes instead of fixed unit-sized jumps. This model finds use in a wide range of fields, including insurance (see \cite{Dickson2016}), evolutionary biology (see \cite{Huelsenbeck2000}), \textit{etc}. For more information on fractional generalizations of the compound Poisson process, we refer the reader to \cite{Beghin2012}, \cite{Khandakar2023}, and the references cited therein.

The Poisson random field (PRF) or spatial Poisson point process describes the random distribution of points within a space. It serves as a key statistical model for analyzing spatial point patterns, where each point corresponds to the location of a specific object. It has application across various fields, including astronomy, ecology, physics, telecommunications, \textit{etc}.

Let $\mathbb{R}_+$ denote the set of non-negative real numbers and  $\mathbb{T}=[0,T_1]\times[0,T_2]$ be a rectangle in $\mathbb{R}^2_+$.
A collection of random variable $\{X(t_1,t_2),\ (t_1,t_2)\in\mathbb{T}\}$ is called a two parameter random process.

For $(s_1,s_2)\preceq (t_1,t_2)$, that is, $s_1\leq t_1$ and $s_2\leq t_2$, the rectangular increment of two parameter random process over rectangle $(s_1,t_1]\times (s_2,t_2]$ is defined as 
\begin{equation*}
	X((s_1,t_1]\times(s_2,t_2])=\Delta_{s_1,s_2}X(t_1,t_2)\coloneqq X(t_1,t_2)-X(s_1,t_2)-X(t_1,s_2)+X(s_1,s_2).
\end{equation*}

Let $(\Omega,\mathcal{F},P)$ be a complete probability space and $\{\mathcal{F}_{t_1,t_2},(t_1,t_2)\in\mathbb{T}\}$ be a collection of sub-sigma algebras of $\mathcal{F}$ such that\\
\noindent (i) $\mathcal{F}_{0,0}$ contains all the null sets in $\mathcal{F}$,\\
\noindent (ii) $\mathcal{F}_{s_1,s_2}\subseteq \mathcal{F}_{t_1,t_2}$ whenever $s_1\leq t_1$ and $s_2\leq t_2$,\\
\noindent (iii) for all $r\in\mathbb{T}$, $\mathcal{F}_r=\cap_{r\prec r'}\mathcal{F}_{r'}$, where $r=(s_1,s_2)\prec r'=(t_1,t_2)$ denotes the partial ordering on $\mathbb{R}^2_+$, that is, $r\prec r'$ implies $s_1<t_1$ and $s_2<t_2$.

The PRF $\{N(t_1,t_2), (t_1,t_2)\in\mathbb{R}^2_+\}$ on plane denoted by
$
N(t_1,t_2)\coloneqq N([0,t_1]\times[0,t_2])$, $ (t_1,t_2)\in\mathbb{R}_+^2
$ is a point process that represent the total number of random points inside the rectangle $[0,t_1]\times[0,t_2]$. 
It has the following characterization (see \cite{Baddeley2007}, \cite{Merzbach1986}): 

The PRF $\{N(t_1,t_2),\ (t_1,t_2)\in\mathbb{R}^2_+\}$ is a c\`adl\`ag field adapted to $\{\mathcal{F}_{t_1,t_2}, (t_1,t_2)\in\mathbb{R}^2_+\}$ such that $N(0,t_2)=N(t_1,0)=0$ with probability one, the increment $N((s_1,t_1]\times(s_2,t_2])$ is a Poisson random variable with mean $\lambda(t_1-s_1)(t_2-s_2)$. 
In particular, for $s_1=s_2=0$, the distribution $p(n,t_1,t_2)=\mathrm{Pr}(N(t_1,t_2)=n)$ of PRF is given by
\begin{equation}\label{prfdist}
	p(n,t_1,t_2)=\frac{e^{-\lambda t_1t_2}(\lambda t_1t_2)^n}{n!},\ n\ge0.
\end{equation}
Moreover, for disjoint rectangles $R_1,R_2,\dots,R_m$ in $\mathbb{R}^2_+$ like $(s_1,t_1]\times (s_2,t_2]$, the random variables $N(R_1),N(R_2)$, $\dots,N(R_m)$ are independent of each other. 

From Eq. (19) of \cite{Kataria2024}, the distribution of PRF solves the following system of differential equations:
\begin{equation}\label{prfeqs}
	\frac{\partial^2}{\partial t_2\partial t_1}p(n,t_1,t_2)=(n+1)\lambda p(n+1,t_1,t_2)-(2n+1)\lambda p(n,t_1,t_2)+n\lambda p(n-1,t_1,t_2),\ n\ge0,
\end{equation}
with $p(0,0,0)=p(0,t_1,0)=p(0,0,t_2)=1$ for all $t_1\ge0$ and $t_2\ge0$.

In \cite{Kataria2024}, a time fractional variants of the PRF, namely, the fractional Poisson random field (FPRF) $\{N_{\alpha_1,\alpha_2}(t_1,t_2),\ (t_1,t_2)\in\mathbb{R}^2_+\}$, $0<\alpha_i\leq1$, $i=1,2$ is introduced and studied. Its distribution $p_{\alpha_1,\alpha_2}(n,t_1,t_2)=\mathrm{Pr}\{N_{\alpha_1,\alpha_2}(t_1,t_2)=n\}$, $n\ge0$ solves
	\begin{equation*}
		\frac{\partial^{\alpha_1+\alpha_2}}{\partial t_2^{\alpha_2}\partial t_1^{\alpha_1}}p_{\alpha_1,\alpha_2}(n,t_1,t_2)=(n+1)\lambda p_{\alpha_1,\alpha_2}(n+1,t_1,t_2)-(2n+1)\lambda p_{\alpha_1,\alpha_2}(n,t_1,t_2)+n\lambda p_{\alpha_1,\alpha_2}(n-1,t_1,t_2),
	\end{equation*}
	with $p_{\alpha_1,\alpha_2}(0,0,0)=p_{\alpha_1,\alpha_2}(0,t_1,0)=p_{\alpha_1,\alpha_2}(0,0,t_2)=1$ for all $t_1\ge0$ and $t_2\ge0$.
	
	Some fractional variants of the PRF are studied in \cite{Vishwakarma2025a}. In \cite{Kataria2024}, a fractional compound Poisson random field on $\mathbb{R}^2_+$, that is,
	$
		\sum_{j=1}^{N_{\alpha_1,\alpha_2}(t_1,t_2)}Y_j,\ (t_1,t_2)\in\mathbb{R}^2_+
$
	is introduced where $Y_j$'s are independent and identically distributed random variables which are independent of the FPRF. In particular, for $\alpha_1=\alpha_2=1$, it reduces to the following compound Poisson random field:
	$
		\sum_{j=1}^{N(t_1,t_2)}Y_j$, $ (t_1,t_2)\in\mathbb{R}^2_+.
	$

In this paper, we study some compound Poisson random fields and analyze the asymptotic distribution of their scaled versions. The outline of the paper is as follows:

In the next section, we collect some known results that will be useful in the paper. In Section \ref{sec3}, we consider the compound Poisson random field (CPRF) on plane and study some of its distributional properties. In particular, we analyzed the asymptotic distributions for some scaled versions of CPRF with normal compounding. In Section \ref{sec4}, we study some particular cases of the time fractional CPRF, namely, with normal, exponential and Mittag-Leffler compounding random variables. We obtain their distributions and the associated governing differential equations. In the case of normal compounding some limiting results are discussed. In Section \ref{sec5}, we introduce and study the space fractional and the space-time fractional variants of PRF. The time-changed representations of these processes in terms of one parameter Poisson process are derived. Moreover, we introduce some different fractional CPRFs using the fractional variants of PRF. Finally, we summarize the paper with the scope of future direction in the concluding remark section. 
\section{Preliminaries}\label{pre} 
Here, we collect some known results that will be used later. First, we recall the definition of two parameter L\'evy process (see \cite{Vares1982}).
\begin{definition}(\textbf{Two parameter L\'evy process})
A real valued two parameter random process $\{X(t_1,t_2),\ (t_1,t_2)\in\mathbb{R}^2_+\}$ is called L\'evy process if \\
\noindent (i) $X(0,t_2)=X(t_1,0)=0$ almost surely,\\
\noindent (ii) for fix $(h,k)\in\mathbb{R}^2_+$,  the distribution of $\Delta_{t_1,t_2}X(t_1+h,t_2+k)$ is independent of the choice of $(t_1,t_2)$. This is known as the stationary increment property,\\
\noindent (iii) for disjoint rectangles $R_1,R_2,\dots,R_m$ in $\mathbb{R}^2_+$ of type $(s_1,t_1]\times (s_2,t_2]$, the random variables $X(R_1),X(R_2)$, $\dots,X(R_m)$ are independent of each other. This property is known as the independent increments property.
\end{definition}
The characteristic function of the two parameter L\'evy process  is given by $\mathbb{E}e^{i\langle\eta, X(t_1,t_2)\rangle}=e^{-t_1t_2\Psi(\eta)}$, $\eta\in\mathbb{R}$. Here, $\Psi$ is the L\'evy exponent of $\{X(t_1,t_2),\ (t_1,t_2)\in\mathbb{R}^2_+\}$ given by
\begin{equation*}
	\Psi(\eta)=ia\eta+\frac{b\eta^2}{2}+\int_{\mathbb{R}^d-\{0\}}(1+i\eta w\textbf{1}_{|w|\leq1}-e^{i\eta w})\phi(\mathrm{d}w),
\end{equation*}
for some constants $a\in{R}$ and $b\ge0$, where $\phi$ is the L\'evy measure on $\mathbb{R}$, that is, 
$
	\phi(-\infty,0)=0$ and $\int_{0}^{\infty}\min\{|x|,1\}\phi(\mathrm{d}w)<\infty$.

\subsection{Two parameter stable subordinator}\label{tpsub} Let $\{H_\alpha(t_1,t_2),\ (t_1,t_2)\in\mathbb{R}^2_+\}$, $0<\alpha<1$ be a two parameter L\'evy process such that for any $(t_1,t_2)\preceq(t_1',t_2')$ in $\mathbb{R}^2_+$, we have $H_\alpha(t_1,t_2)\leq H_\alpha(t_1',t_2')$. It is called a two parameter stable subordinator of index $\alpha$ if its Laplace transform is given by
\begin{equation}\label{tpsublap}
\mathbb{E}e^{-uH_\alpha(t_1,t_2)}=e^{-t_1t_2u^\alpha},\ u>0,\, (t_1,t_2)\in\mathbb{R}^2_+.
\end{equation}

\subsection{Inverse stable subordinator}\label{sub} A one parameter stable subordinator $\{H_\alpha(t),\ t\ge0\}$ is a non-decreasing L\'evy process whose Laplace transform is given by
\begin{equation}\label{sublap}
	\mathbb{E}e^{-uH_\alpha(t)}=e^{-tu^\alpha},\ u>0.
\end{equation}
Its first passage time process $\{L_\alpha(t),\ t\ge0\}$, that is,
\begin{equation}\label{insub}
	L_\alpha(t)\coloneqq\inf\{s\ge0:H_\alpha(s)\ge t\}
\end{equation}
is called the inverse $\alpha$-stable subordinator.
Its density has the following Laplace transform:
\begin{equation}\label{insublap}
	\int_{0}^{\infty}e^{-zt}\mathrm{Pr}\{L_\alpha(t)\in\mathrm{d}x\}=z^{\alpha-1}e^{-z^\alpha x}\, z>0.
\end{equation}
\subsection{Some special functions}
\subsubsection{Mittag-Leffler function} For $\alpha>0$, $\beta\in\mathbb{R}$ and $\gamma\in\mathbb{R}$, the three parameter Mittag-Leffler function is defined as follows (see \cite{Kilbas2006}):
\begin{equation}\label{tpml}
	E_{\alpha,\beta}^\gamma(x)\coloneqq\sum_{k=0}^{\infty}\frac{\gamma^{(k)}x^k}{\Gamma(\alpha k+\beta)k!},\ x\in\mathbb{R},
\end{equation}
where $\gamma^{(k)}=\gamma(\gamma+1)\dots(\gamma+k-1)$.

The following result will be used (see \cite{Kilbas2006}):
\begin{equation}\label{tpmllap}
	\int_{0}^{\infty}e^{-zt}t^{\beta-1}E^{\gamma}_{\alpha,\beta}(c t^\alpha)\,\mathrm{d}t=\frac{z^{\alpha\gamma-\beta}}{(z^\alpha-c)^\gamma},\ |cz^{-\alpha}|<1,\ c\in\mathbb{R}.
\end{equation}

For $\beta=\gamma=1$, (\ref{tpml}) reduces to the one parameter Mittag-Leffler function defined by
$	E_{\alpha,1}(x)
$.
Its Laplace transform is given by
\begin{equation}\label{opmllap}
	\int_{0}^{\infty}e^{-zt}E_{\alpha,1}(ct)=\frac{z^{\alpha-1}}{z^\alpha-c},\ |cz^{-\alpha}|<1,\ c\in\mathbb{R}.
\end{equation}

\subsubsection{Generalized Wright function}For real numbers $a_i$'s, $b_j$'s and non zero real numbers $\alpha_i$'s, $\beta_j$'s, $i\in\{1,2,\dots,m\}$, $j\in\{1,2,\dots,l\}$, the generalized Wright function is defined as follows (see \cite{Kilbas2006}, Eq. (1.11.14)):
\begin{align}\label{genwrit}
	{}_m\Psi_l\left[\begin{matrix}
		(a_i,\alpha_i)_{1,m}\\\\
		(b_j,\beta_j)_{1,l}
	\end{matrix}\Bigg| x \right]&={}_m\Psi_l\left[\begin{matrix}
		(a_1,\alpha_1),\,(a_2,\alpha_2),\dots,(a_m,\alpha_m)\\\\
		(b_1,\beta_1),\,(b_2,\beta_2),\dots,(b_l,\beta_l)
	\end{matrix}\Bigg| x \right]\nonumber\\
	&\coloneqq\sum_{n=0}^{\infty}\frac{\prod_{i=1}^{m}\Gamma(a_i+n\alpha_i)x^n}{\prod_{j=1}^{l}\Gamma(b_j+n\beta_j)n!},\ x\in\mathbb{R}.
\end{align}
\subsection{Caputo fractional derivative} For an appropriate function $g(\cdot)$, its Caputo fractional derivative of order $\alpha\in(0,1]$ is defined as follows (see \cite{Kilbas2006}):
\begin{equation}\label{caputoder}
	\frac{\mathrm{d}^\alpha}{\mathrm{d}t^\alpha}g(t)=\begin{cases}
		\frac{1}{\Gamma(1-\alpha)}\int_{0}^{t}\frac{f'(s)}{(t-s)^{\alpha}}\,\mathrm{d}s,\ 0<\alpha<1,\\
		f'(t),\ \alpha=1.
	\end{cases}
\end{equation}
Its Laplace transform is given by
\begin{equation}\label{caputoderlap}
	\int_{0}^{\infty}e^{-zt}\frac{\mathrm{d}^\alpha}{\mathrm{d}t^\alpha}g(t)\,\mathrm{d}t=z^{\alpha}\int_{0}^{\infty}e^{-zt}g(t)\,\mathrm{d}t-z^{\alpha-1}g(0^+),\ z>0.
\end{equation}
\section{Compound Poisson random field}\label{sec3}
 Let $\{Y_j\}_{j\ge1}$ be a sequence of independent and identically distributed (iid) random variables, and $S_n=\sum_{j=1}^{n}Y_j$, $n\ge0$ with $S_0=0$. Also, let $\{N(t_1,t_2),\ (t_1,t_2)\in\mathbb{R}^2_+\}$ be a two parameter Poisson random field with parameter $\lambda>0$ such that it is independent of $\{Y_j\}_{j\ge1}$. 
 
 We consider a two parameter random process $\{X(t_1,t_2),\ (t_1,t_2)\in\mathbb{R}^2_+\}$, namely, the compound Poisson random field (CPRF) defined as follows:
\begin{equation}\label{cprf}
	X(t_1,t_2)\coloneqq S_{N(t_1,t_2)}=\sum_{j=1}^{N(t_1,t_2)}Y_j,\ (t_1,t_2)\in\mathbb{R}^2_+.
\end{equation} 

Let $F(\cdot)$ be the common distribution function of $Y_j$'s. Then, the distribution function of CPRF is given by (see \cite{Kataria2024})
\begin{equation*}
	\mathrm{Pr}\{X(t_1,t_2)\leq x\}=\Theta(x)+\sum_{n=1}^{\infty}\frac{(-\lambda t_1t_2)^n}{n!}\bigg(\Theta(x)+\sum_{k=1}^{n}\binom{n}{k}(-1)^kF^{*k}(x)\bigg),\ x\in\mathbb{R},
\end{equation*}
where $\Theta(\cdot)$ is the Heaviside function and $F^{*k}$ denotes the $k$-fold convolution of $F$.
For a fractional variant of the CPRF, we refer the reader to \cite{Kataria2024}.

Next, we show that the CPRF has stationary and independent increments property.
\begin{proposition}\label{prop}
	The CPRF vanishes on axes almost surely, and it has stationary and independent rectangular increments.
\end{proposition}

\begin{proof}
	As $N(0,t_2)=N(t_1,0)=0$ almost surely, we have $X(0,t_2)=X(t_1,0)=0$ almost surely. Now, for $(h,k)\in\mathbb{R}^2_+$ and $(t_1,t_2)\in\mathbb{R}^2_+$, it is sufficient to show that the distribution of $\Delta_{t_1,t_2}X(t_1+h,t_2+k)$ is not a function of $(t_1,t_2)$. 	
	Let $f(\cdot;\Delta_{s_1,s_2}X(t_1,t_2))$ and $f(\cdot;S_n)$  be the densities of $\Delta_{s_1,s_2}X(t_1,t_2)$ and $S_n$, respectively, and let $p(n_1,n_2,n_3,n_4)=\mathrm{Pr}\{N(s+h,t+k)=n_1,N(s+h,t)=n_2,N(s,t+k)=n_3,N(s,t)=n_4\}$, $(n_1,n_2,n_3,n_4)\in\Theta(n_1,n_2,n_3,n_4)$, where $\Theta(n_1,n_2,n_3,n_4)=\{(n_1,n_2,n_3,n_4)\in\mathbb{N}_0^4:n_1\ge n_j, j=2,3,4,\ n_2\ge n_4,\ \text{and}\ n_3\ge n_4\}$. Then, we have
	{\small\begin{align*}	f(x;\Delta_{s,t}X(s+h,t+k))\mathrm{d}x&=\sum_{\Theta(n_1,n_2,n_3,n_4)}p(n_1,n_2,n_3,n_4)\mathrm{Pr}\biggl\{\bigg(\sum_{j=1}^{n_1}Y_j-\sum_{j=1}^{n_2}Y_j-\sum_{j=1}^{n_3}Y_j+\sum_{j=1}^{n_4}Y_j\bigg)\in\mathrm{d}x\biggr\}\\
		&=\sum_{\substack{\Theta(n_1,n_2,n_3,n_4)\\n_2\ge n_3}}p(n_1,n_2,n_3,n_4)\mathrm{Pr}\bigg\{\bigg(\sum_{j=n_2+1}^{n_1}Y_j-\sum_{j=n_4+1}^{n_3}Y_j\bigg)\in\mathrm{d}x\bigg\}\\
		&\ \ +\sum_{\substack{\Theta(n_1,n_2,n_3,n_4)\\n_2<n_3}}p(n_1,n_2,n_3,n_4)\mathrm{Pr}\biggl\{\bigg(\sum_{j=n_3+1}^{n_1}Y_j-\sum_{j=n_4+1}^{n_2}Y_j\bigg)\in\mathrm{d}x\biggr\}\\
		&=\sum_{\substack{\Theta(n_1,n_2,n_3,n_4)\\n_2\ge n_3}}p(n_1,n_2,n_3,n_4)\mathrm{d}x\int_{-\infty}^{\infty}f(x+y;S_{n_1-n_2})f(y;S_{n_3-n_4})\,\mathrm{d}y\\
		&\ \ +\sum_{\substack{\Theta(n_1,n_2,n_3,n_4)\\n_2<n_3}}p(n_1,n_2,n_3,n_4)\mathrm{d}x\int_{-\infty}^{\infty}f(x+y;S_{n_1-n_3})f(y;S_{n_2-n_4})\,\mathrm{d}y\\
		&=\sum_{\substack{\Theta(n_1,n_2,n_3,n_4)\\n_2\ge n_3}}p(n_1,n_2,n_3,n_4)\mathrm{Pr}\{S_{n_1-n_2}-S_{n_3-n_4}\in\mathrm{d}x\}\\
		&\ \ +\sum_{\substack{\Theta(n_1,n_2,n_3,n_4)\\n_2<n_3}}p(n_1,n_2,n_3,n_4)\mathrm{Pr}\{S_{n_1-n_3}-S_{n_2-n_4}\in\mathrm{d}x\}\\
		&=\sum_{\substack{\Theta(n_1,n_2,n_3,n_4)\\n_2\ge n_3}}p(n_1,n_2,n_3,n_4)\mathrm{Pr}\{S_{n_1-n_2-n_3+n_4}\in\mathrm{d}x\}\\
		&\ \ +\sum_{\substack{\Theta(n_1,n_2,n_3,n_4)\\n_2<n_3}}p(n_1,n_2,n_3,n_4)\mathrm{Pr}\{S_{n_1-n_3-n_2+n_4}\in\mathrm{d}x\}\\
		&=\sum_{\Theta(n_1,n_2,n_3,n_4)}p(n_1,n_2,n_3,n_4)\mathrm{Pr}\{S_{n_1-n_2-n_3+n_4}\in\mathrm{d}x\}\\
		&=\mathrm{Pr}\bigg(\sum_{j=1}^{\Delta_{s,t}N(s+h,t+k)}Y_j\in\mathrm{d}x\bigg).
	\end{align*}}
	Thus, the stationary increments property of the CPRF follows from the stationary increments property of the Poisson random field.	
	To establish its independent increments property, we proceed as follows:
	For $j=1,2,\dots,m$, consider the disjoint rectangles $(s_j,t_j]\times(s_{j+1},t_{j+1}]$ in $\mathbb{R}^2_+$. Then, the increments over these rectangles are given by
	\begin{equation}\label{statinpf}
		\Delta_{s_j,t_j}X(s_{j+1},t_{j+1})=\sum_{l=1}^{N(s_{j+1},t_{j+1})}Y_l-\sum_{l=1}^{N(s_{j},t_{j+1})}Y_l-\sum_{l=1}^{N(s_{j+1},t_{j})}Y_l+\sum_{l=1}^{N(s_{j},t_{j})}Y_l.
	\end{equation}
	 Note that the indices $l=1,2,\dots$ of $Y_l$'s in  (\ref{statinpf}) refer to the points in disjoint rectangles $(s_j,t_j]\times(s_{j+1},t_{j+1}]$, $j\in\{1,2,\dots,m\}$. Thus, the independent increments property of CPRF follows from the independence of $Y_l$'s and independent increments property of the PRF. This completes the proof. 
\end{proof}

\begin{remark}
	Let $C_b(\mathbb{R})$ be the space of continuous and bounded function on $\mathbb{R}$. For $f\in C_b(\mathbb{R})$, let us consider the operator
	$
		T_{t_1,t_2}f(x)=\mathbb{E}f(x+X(t_1,t_2))
	$, $(t_1,t_2)\in\mathbb{R}^2_+$ 
	where $\{X(t_1,t_2),\ (t_1,t_2)\in\mathbb{R}^2_+\}$ is the CPRF. Then, for any $s_1\ge0$, we have
	\begin{align*}
		T_{t_1+s_1,t_2}f(x)&=\mathbb{E}f(x+X(t_1+s_1,t_2))\\
		&=\mathbb{E}f(x+\Delta_{t_1,0}X(t_1+s_1,t_2)+X(t_1,t_2))\\
		&=\mathbb{E}f(x+X(s_1,t_2)+X(t_1,t_2))\\
		&=T_{t_1,t_2}T_{s_1,t_2}f(x),
	\end{align*}
	where we have used the stationary increment property of CPRF to get the penultimate step. Similarly, it can be shown that $T_{t_1,t_2+s_2}f(x)=T_{t_1,t_2}T_{t_1,s_2}f(x)$ for all $s_2\ge0$. Thus, $T_{t_1,t_2}$ is  a coordinatewise semigroup operator.
\end{remark}

\subsection{CPRF with normal compounding} 
Let the sequence $\{Y_j\}_{j\ge1}$ in (\ref{cprf}) be the sequence of standard normal random variables, say, $\{Z_j\}_{j\ge1}$, and the rate of $\{N(t_1,t_2),\ (t_1,t_2)\in\mathbb{R}^2_+\}$ be $1$. Then, we call the CPRF as the normal CPRF and denote it by $\{\mathscr{X}(t_1,t_2),\ (t_1,t_2)\in\mathbb{R}^2_+\}$. Thus, $\mathscr{X}(t_1,t_2)=\sum_{j=1}^{N(t_1,t_2)}Z_j$. 

In view of Proposition \ref{prop}, for $(s_1,t_1)\prec(s_2,t_2)$ in $\mathbb{R}^2_+$, the characteristic function of $\Delta_{s_1,t_1}X(s_2,t_2)$ is given by
\begin{align}
	\mathbb{E}e^{iu\Delta_{s_1,t_1}\mathscr{X}(s_2,t_2)}&=\mathbb{E}\bigg(\mathbb{E}\bigg(\exp\bigg(iu\sum_{j=1}^{\Delta_{s_1,t_1}N(s_2,t_2)}Z_j\bigg)\bigg|\Delta_{s_1,t_1}N(s_2,t_2)\bigg)\bigg)\nonumber\\
	&=\mathbb{E}\exp\bigg(-\frac{u^2\Delta_{s_1,t_1}N(s_2,t_2)}{2}\bigg)\nonumber\\
	&=\exp\big((s_2-s_1)(t_2-t_1)(e^{-u^2/2}-1)\big),\ u\in\mathbb{R}.\label{chdx}
\end{align}

The following result provides the auto covariance function of the normal CPRF.
\begin{proposition}\label{prop1}
	For $(s_1,t_1)$ and $(s_2,t_2)$ in $\mathbb{R}^2_+$, we have the following joint moment:
	\begin{equation}\label{jchdx}
		\mathbb{E}\mathscr{X}(s_1,t_1)\mathscr{X}(s_2,t_2)=(s_1\wedge s_2)(t_1\wedge t_2),
	\end{equation}
	where $s\wedge t$ denotes the minimum of $s$ and $t$.
\end{proposition}
\begin{proof}
	Let us consider the joint characteristic function $\phi(u,v)=\mathbb{E}\exp(iu\mathscr{X}(s_1,t_1)+iv\mathscr{X}(s_2,t_2))$, $u\in\mathbb{R}$, $v\in\mathbb{R}$. Then, 
	\begin{equation*}
		\mathbb{E}\mathscr{X}(s_1,t_1)\mathscr{X}(s_2,t_2))=-\frac{\partial^2}{\partial u\partial v}\phi(u,v)\bigg|_{u=v=0}.
	\end{equation*}
	For $(s_1,t_1)$ and $(s_2,t_2)$ in $\mathbb{R}^2_+$, we have the following four possible ordering in the plane:\\
	$1.\ s_1<s_2,\ t_1<t_2$,\\
	$2.\ s_1<s_2,\ t_1>t_2$,\\
	$3.\ s_1>s_2,\ t_1<t_2$,\\
	$4.\  s_1>s_2,\ t_1>t_2$.
	
	Note that first two cases are sufficient to derive (\ref{jchdx}). 
	\paragraph{\textbf{Case I}} If $s_1<s_2$ and $t_1<t_2$ then there are four disjoint rectangles as shown in Figure \ref{fig1}, and $\mathscr{X}(s_2,t_2)$ can be written in terms of the increments over these rectangles as follows: $\mathscr{X}(s_2,t_2)=\Delta_{s_1,t_1}\mathscr{X}(s_2,t_2)+\Delta_{0,t_1}\mathscr{X}(s_1,t_2)+\Delta_{s_1,0}\mathscr{X}(s_2,t_1)+\mathscr{X}(s_1,t_1)$. Therefore,
	
	\begin{figure}[ht!]
		\includegraphics[width=5cm]{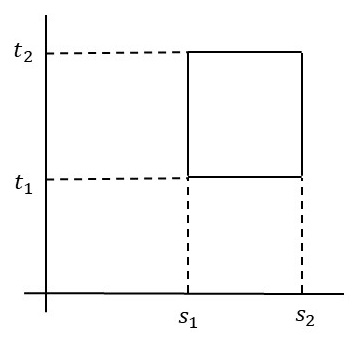}
		\caption{First type of ordering in plane}\label{fig1}
	\end{figure}
	
	\begin{align}
		\phi(u,v)
		&=\mathbb{E}(i(u+v)\mathscr{X}(s_1,t_1)+iv\Delta_{s_1,t_1}\mathscr{X}(s_2,t_2)+iv\Delta_{0,t_1}\mathscr{X}(s_1,t_2)+iv\Delta_{s_1,0}\mathscr{X}(s_2,t_1))\nonumber\\
		&=\mathbb{E}(i(u+v)\mathscr{X}(s_1,t_1))\mathbb{E}\exp(iv\Delta_{s_1,t_1}\mathscr{X}(s_2,t_2))\mathbb{E}(iv\Delta_{0,t_1}\mathscr{X}(s_1,t_2))\mathbb{E}(iv\Delta_{s_1,0}\mathscr{X}(s_2,t_1))\nonumber\\
		&=\exp\big(s_1t_1(e^{-(u+v)^2/2}-1)\big)\exp\big((s_2-s_1)(t_2-t_1)(e^{-v^2/2}-1)\big)\nonumber\\
		&\hspace{3cm} \cdot\exp\big(s_1(t_2-t_1)(e^{-v^2/2}-1)\big)\exp\big((s_2-s_1)t_1(e^{-v^2/2}-1)\big)\nonumber\\
		&=\exp\big(s_1t_1(e^{-(u+v)^2/2}-1)+(s_2t_2-s_1t_1)(e^{-v^2/2}-1)\big),\label{jmpf1}
	\end{align}
	where the second equality follows from the independent increments property of CPRF and to get the penultimate step, we have used (\ref{chdx}). From (\ref{jmpf1}), we have
	\begin{align*}
		\frac{\partial^2}{\partial u\partial v}\phi(u,v)&=\phi(u,v)(-s_1t_1(u+v)e^{-(u+v)^2/2}-(s_2t_2-s_1t_1)ve^{-v^2/2})(-s_1t_1(u+v)e^{-(u+v)^2/2})\\
		&\ \ +\phi(u,v)(-s_1t_1e^{-(u+v)^2/2}+s_1t_1(u+v)^2e^{-(u+v)^2/2}).
	\end{align*}
	Thus, $\mathbb{E}\mathscr{X}(s_1,t_1)\mathscr{X}(s_2,t_2))=s_1t_1$ whenever $s_1<s_2$ and $t_1<t_2$.
	\paragraph{\textbf{Case II}} If $s_1<s_2$ and $t_1>t_2$ then we have three disjoint rectangles as shows in Figure \ref{fig2}, and $\mathscr{X}(s_1,t_1)=\Delta_{0,t_2}\mathscr{X}(s_1,t_1)+\mathscr{X}(s_1,t_2)$ and $\mathscr{X}(s_2,t_2)=\Delta_{s_1,0}\mathscr{X}(s_2,t_2)+\mathscr{X}(s_1,t_2)$. So, by using the independent increments property of CPRF, we get
	
	\begin{figure}[ht!]
		\includegraphics[width=5cm]{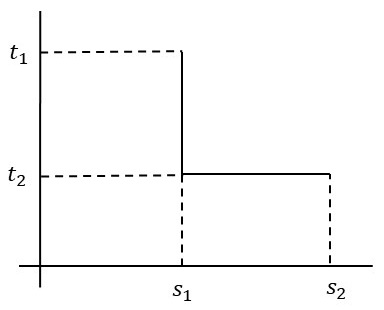}
		\caption{Second type of ordering in plane}\label{fig2}
	\end{figure}
	\begin{align*}
		\phi(u,v)
		&=\mathbb{E}\exp(iu\Delta_{0,t_2}\mathscr{X}(s_1,t_1)+i(u+v)\mathscr{X}(s_1,t_2)+iv\Delta_{s_1,0}\mathscr{X}(s_2,t_2))\\
		&=\exp(s_1(t_1-t_2)(e^{-u^2/2}-1))\exp(s_1t_2(e^{-(u+v)^2/2}-1))\exp((s_2-s_1)t_2(e^{-v^2/2}-1)).
	\end{align*}
	So,
	\begin{align*}
		\frac{\partial^2}{\partial u\partial v}\phi(u,v)&=\phi(u,v)(-s_1t_2(u+v)e^{-(u+v)^2/2}-(s_2-s_1)t_2ve^{-v^2/2}) (-s_1(t_1-t_2)ue^{-u^2/2}\\
		&\ \ -s_1t_2(u+v)e^{-(u+v)^2/2}) +\phi(u,v)(-s_1(t_1-t_2)ue^{-u^2/2}-s_1t_2e^{-(u+v)^2/2}).
	\end{align*}
	Hence,
	$	\mathbb{E}\mathscr{X}(s_1,t_1)\mathscr{X}(s_2,t_2))=s_1t_2$ whenever $s_1<s_2$ and $t_1>t_2$.
	
	Thus, for general case, we have
	\begin{equation*}
		\mathbb{E}\mathscr{X}(s_1,t_1)\mathscr{X}(s_2,t_2))=\begin{cases}
			s_1t_1,\ s_1<s_2,\ t_1<t_2,\\
			s_1t_2,\ s_1<s_2,\ t_1>t_2,\\
			s_2t_1,\ s_1>s_2,\ t_1<t_2,\\
			s_2t_2,\ s_1>s_2,\ t_1>t_2.
		\end{cases}
	\end{equation*}
	This completes the proof.
\end{proof}
	
	\subsubsection{Limiting result} Let $\{Z_j\}_{j\ge1}$ be a sequence of iid standard normal random variables and $\{N(t),\ t\ge0\}$ be the Poisson process with unit transition rate such that it is independent of $\{Z_j\}_{j\ge1}$. Then, the following scaled continuous time random walk:
	\begin{equation}\label{cpp}
		\bigg\{r^{1/2}\sum_{j=1}^{N(t/r)}Z_j,\ t\ge0\bigg\},
	\end{equation}
	 converges to the one parameter Brownian motion $\{B(t),\ t\ge0\}$ in an appropriate Skorokhod topology as $r\to0$.
	 
	 For the following definition of two parameter Brownian sheet, we refer the reader to \cite{Bardina2000}.
	\begin{definition}\label{BS}
		A two parameter random process $\{W(t_1,t_2),\ (t_1,t_2)\in\mathbb{R}^2_+\}$ is called Brownian sheet if $W(0,t_2)=W(t_1,0)=0$ almost surely, and it has independent and stationary rectangular increments. Also, for $(s_1,t_1)\prec(s_2,t_2)$, the increment $\Delta_{s_1,t_1}W(s_2,t_2)$ has normal distribution with zero mean and variance $(s_2-s_1)(t_2-t_1)$.
	\end{definition}
	Let us define a scaled normal CPRF as follows:
	\begin{equation}\label{scl}
		\mathscr{X}_n(t_1,t_2)=n^{-1}\sum_{j=1}^{N(nt_1,nt_2)}Z_j.
	\end{equation}
	As $n\to\infty$, its characteristic function, that is, $\mathbb{E}e^{iu\mathscr{X}_n(t_1,t_2)}=\exp\big(n^{2}t_1t_2(e^{-u^2n^{-2}/2}-1)\big)$  converges to $\exp(-u^2t_1t_2/2)$, the characteristic function of a Gaussian random variable with mean zero and variance $t_1t_2$. Thus, for fix $(t_1,t_2)\in\mathbb{R}^2_+$, the sequence $\{\mathscr{X}_n(t_1,t_2)\}_{n\ge1}$ converges in distribution to $W(t_1,t_2)$.
	
	For any set of constants $c_1,c_2,\dots,c_m$ and $(s_1,t_1),(s_2,t_2),\dots$, $(s_m,t_m)\in\mathbb{R}^2_+$, let us consider $U_n=\sum_{j=1}^{m}c_j\mathscr{X}_n(s_j,t_j)$.
	Then, $\mathbb{E}U_n=0$ and
	\begin{align*}
		\mathbb{E}U_n^2=\mathbb{E}\bigg(\sum_{j=1}^{m}c_j\mathscr{X}_n(s_j,t_j)\bigg)^2&=\sum_{j=1}^{m}\sum_{l=1}^{m}c_jc_l\mathbb{E}\mathscr{X}_n(s_j,t_j)\mathscr{X}_n(s_l,t_l)\\
		&=n^{-2}\sum_{j=1}^{m}\sum_{l=1}^{m}c_jc_ln^2(s_j\wedge t_j)(s_{j+1}\wedge t_{j+1})\\
		&=\mathbb{E}\bigg(\sum_{j=1}^{m}c_jW(s_j,t_j)\bigg)^2,
	\end{align*} 
	where we have used Proposition \ref{prop1} to get the penultimate step. 
	\begin{remark}		
		The above limiting result for scaled continuous time random walk (\ref{cpp}) is expected for the scaled normal CPRF (\ref{scl}). However, due to limitation of technique for the functional convergence of two parameter random processes, we could show the convergence of scaled CPRF in distribution only for a fixed $(t_1,t_2)\in\mathbb{R}^2_+$.
	\end{remark} 

\section{Fractional variants of CPRF}\label{sec4}
In this section, we introduce and study some fractional variants of CPRF. 
 
For $\alpha_i\in(0,1]$, $i=1,2$, let $\{N_{\alpha_1,\alpha_2}(t_1,t_2)\in\mathbb{R}^2_+\}$ be the fractional Poisson random field (FPRF) on plane with parameter $\lambda>0$ as defined in \cite{Kataria2024}. Its distribution is given by
\begin{equation}\label{fprfdist}
	p_{\alpha_1,\alpha_2}(n,t_1,t_2)=\mathrm{Pr}\{N_{\alpha_1,\alpha_2}(t_1,t_2)=n\}=\sum_{k=n}^{\infty}\frac{(-1)^{k-n}(k)_{k-n}(k)_n(\lambda t_1^{\alpha_1}t_2^{\alpha_2})^k}{\Gamma(k\alpha_1+1)\Gamma(k\alpha_2+1)},\ n\ge0,
\end{equation} 
where $(k)_n=k(k-1)\dots(k-n+1)$ with $(k)_0=1$. Also, its Laplace transform is 
\begin{equation}\label{fprflap}
	\mathbb{E}\exp(-u{N_{\alpha_1,\alpha_2}(t_1,t_2)})=\sum_{k=0}^{\infty}\frac{k!(\lambda t_1^{\alpha_1}t_2^{\alpha_2}(e^{-u}-1))^k}{\Gamma(k\alpha_1+1)\Gamma(k\alpha_2+1)},\ u>0.
\end{equation}

Let us consider a time fractional CPRF defined as follows (see \cite{Kataria2024}):
\begin{equation}\label{fcprf}
	\mathscr{Y}_{\alpha_1,\alpha_2}(t_1,t_2)=\sum_{j=1}^{N_{\alpha_1,\alpha_2}(t_1,t_2)}Y_j,\ (t_1,t_2)\in\mathbb{R}^2_+,
\end{equation}
where $Y_j$'s are iid random variables that are independent of $\{N_{\alpha_1,\alpha_2}(t_1,t_2)\in\mathbb{R}^2_+\}$. For $\alpha_1=\alpha_2=1$, the time fractional CPRF reduces to the CPRF defined in (\ref{cprf}). The cumulative distribution function of (\ref{fcprf}) is given by
 \begin{equation*}
 	\mathrm{Pr}\{\mathscr{Y}_{\alpha_1,\alpha_2}(t_1,t_2)\leq y\}=\sum_{n=0}^{\infty}p_{\alpha_1,\alpha_2}(n,t_1,t_2)F_Y^{*n}(y),
 \end{equation*}
 where $F_Y^{*n}(y)$ is the $n$-fold convolution of the distribution function of $Y_1$  with $F_Y^{*0}(y)$ being the Heaviside function. Thus, its density function is
 \begin{equation}\label{fcprfden}
 	f_{\mathscr{Y}_{\alpha_1,\alpha_2}}(y,t_1,t_2)=\sum_{n=0}^{\infty}p_{\alpha_1,\alpha_2}(n,t_1,t_2)f_Y^{*n}(y),
 \end{equation}
 where $f_Y^{*n}(y)$ is the density corresponding to $F_Y^{*n}(y)$.
 
\subsection{Time fractional CPRF with normal compounding} Let us consider the following time fractional CPRF:
\begin{equation}\label{case1def}
	\mathscr{X}_{\alpha_1,\alpha_2}(t_1,t_2)\coloneqq\sum_{j=1}^{N_{\alpha_1,\alpha_2}^{(1)}(t_1,t_2)}Z_j,\ (t_1,t_2)\in\mathbb{R}^2_+,\ 0<\alpha_i\leq1,\ i=1,2,
\end{equation}
where $N_{\alpha_1,\alpha_2}^{(1)}(t_1,t_2)$ is FPRF with parameter $1$ that is independent of the iid standard normal random variables $Z_j$'s. In this case, the $n$-fold convolution of distribution function of $Z_j$ is given by
$
	F_Z^{*n}=(1+\text{erf}({y}/{\sqrt{2n}}))/2$,
where 
$\text{erf}(y)=(\sqrt{\pi})^{-1}\int_{0}^{y}e^{-x^2}\,\mathrm{d}x$.

\subsubsection{Limiting results} 
\paragraph{\textbf{\textit{Case I}}} For $\alpha_2=1$, by using (\ref{fprflap}), the characteristic function of the scaled time fractional CPRF
\begin{equation*}
	\mathscr{X}^{(n)}_{\alpha,1}(t_1,t_2)\coloneqq n^{-1/2}\sum_{j=1}^{N_{\alpha,1}^{(1)}(t_1,nt_2)}Z_j
\end{equation*}
 is given by
 \begin{equation*}
 	\mathbb{E}\exp(iu\mathscr{X}^{(n)}_{\alpha,1}(t_1,t_2))=\sum_{k=0}^{\infty}\frac{(n t_1^{\alpha_1}t_2(e^{-u^2n^{-1}/2}-1))^k}{\Gamma(k\alpha+1)}=E_{\alpha,1}(n t_1^{\alpha}t_2(e^{-u^2n^{-1}/2}-1)).
 \end{equation*}
So,
 \begin{equation}\label{tcbslap}
 	\lim_{n\to\infty}E_{\alpha,1}(n t_1^{\alpha}t_2(e^{-u^2n^{-1}/2}-1))=E_{\alpha,1}(-t_1^{\alpha}t_2u^2/2),\ u\in\mathbb{R},
 \end{equation}
which is a well defined characteristic function.
 
 Let $\{W(t_1,t_2),\ (t_1,t_2)\in\mathbb{R}^2_+\}$ be the two parameter Brownian sheet (see Definition \ref{BS}), and let $\{L_\alpha(t),\ t\ge0\}$, $0<\alpha<1$ be an inverse stable subordinator as defined in (\ref{insub})  such that it is independent of $\{W(t_1,t_2),\ (t_1,t_2)\in\mathbb{R}^2_+\}$. Then, the Fourier transform of the time-changed two parameter process
 \begin{equation}\label{tcbs}
 	W(L_\alpha(t_1),t_2),\ (t_1,t_2)\in\mathbb{R}^2_+,
 \end{equation} 
 is given by
 \begin{align*}
 	\mathbb{E}\exp(iu	W(L_\alpha(t_1),t_2))&=\int_{0}^{\infty}\mathbb{E}e^{iuW(x,t_2)}\mathrm{Pr}\{L_\alpha(t_1)\in\mathrm{d}x\},\ u>0,\\
 	&=\int_{0}^{\infty}e^{-u^2xt_2/2}\mathrm{Pr}\{L_\alpha(t_1)\in\mathrm{d}x\}.
 \end{align*}
 Its Laplace transform is given by
 \begin{align}
 	\int_{0}^{\infty}e^{-zt_1}\mathbb{E}\exp(-u	W(L_\alpha(t_1),t_2))\,\mathrm{d}t_1&=z^{\alpha-1}\int_{0}^{\infty}e^{-(u^2t_2/2+z^\alpha)x}\,\mathrm{d}x\nonumber\\
 	&=\frac{z^{\alpha-1}}{(u^2t_2/2)+z^\alpha},\ z>0,\label{lapbs}
 \end{align}
 where we have used (\ref{insublap}). In view of (\ref{opmllap}), the inverse Laplace transform of (\ref{lapbs}) coincides with (\ref{tcbslap}).  
 Thus, for each $(t_1,t_2)\in\mathbb{R}^2_+$,  $\mathscr{X}_{\alpha,1}^{(n)}(t_1,t_2)$ converges in distribution to (\ref{tcbs}) as $n\to\infty$.
Similarly, 
 	$
 			\mathscr{X}^{(n)}_{1,\alpha}(t_1,t_2)\coloneqq n^{-1/2}\sum_{j=1}^{N_{1,\alpha}^{(1)}(nt_1,t_2)}Z_j
 	$
 	converges in distribution to $W(t_1,L_{\alpha}(t_2))$ as $n\to\infty$.

 \paragraph{\textbf{\textit{Case II}}} Let $\{N_{\alpha_1,\alpha_2}^{(n^2)}(t_1,t_2),\ (t_1,t_2)\in\mathbb{R}^2_+\}$ be the FPRF with parameter $n^2>0$ that is independent of standard normal random variables $Z_1,Z_2,\dots$. Let us consider the following scaled time fractional CPRF:
 \begin{equation}\label{bstc2}
 	\tilde{\mathscr{X}}_{\alpha_1,\alpha_2}^{(n)}(t_1,t_2)\coloneqq n^{-1}\sum_{j=1}^{N_{\alpha_1,\alpha_2}^{(n^2)}(t_1,t_2)}Z_j,\ (t_1,t_2)\in\mathbb{R}^2_+.
 \end{equation}
 Its Fourier transform is given by
 \begin{equation*}
 	\mathbb{E}\exp(iu\tilde{\mathscr{X}}_{\alpha_1,\alpha_2}^{(n)}(t_1,t_2))=\sum_{k=0}^{\infty}\frac{k!(n^2 t_1^{\alpha_1}t_2^{\alpha_2}(e^{-u^2n^{-2}/2}-1))^k}{\Gamma(k\alpha_1+1)\Gamma(k\alpha_2+1)}.
 \end{equation*}
 So,
 \begin{align}\label{bstclap2}
 	\lim_{n\to\infty}\mathbb{E}\exp(iu\tilde{\mathscr{X}}_{\alpha_1,\alpha_2}^{(n)}(t_1,t_2))&=\sum_{k=0}^{\infty}\frac{k!(- t_1^{\alpha_1}t_2^{\alpha_2}u^2/2)^k}{\Gamma(k\alpha_1+1)\Gamma(k\alpha_2+1)}\nonumber\\
 	&={}_2\Psi_2\Bigg[\begin{matrix}
 		(1,1),&(1,1)\\\\
 		(1,\alpha_1),&(1,\alpha_2)
 	\end{matrix}\bigg|-t_1^{\alpha_1}t_2^{\alpha_2}u^2/2\Bigg],\ u\in\mathbb{R},
 \end{align}
 where ${}_2\Psi_2$ is the generalized Wright function defined in (\ref{genwrit}).
 
 For  $i=1,2$, let $\{L_{\alpha_i}(t),\ t\ge0\}$, $0<\alpha_i<1$ be two independent inverse stable subordinator that are independent of the Brownian sheet $\{W(t_1,t_2),\ (t_1,t_2)\in\mathbb{R}^2_+\}$. Let us consider the following time changed Brownian sheet:
 \begin{equation}\label{bstc1}
 	W(L_{\alpha_1}(t_1),L_{\alpha_2}(t_2)),\ (t_1,t_2)\in\mathbb{R}^2_+,
 \end{equation}
 whose Fourier transform is given by
 \begin{align*}
 	\mathbb{E}(iuW(L_{\alpha_1}(t_1),L_{\alpha_2}(t_2)))&=\iint_{\mathbb{R}_+\times\mathbb{R}_+}e^{iuW(x_1,x_2)}\mathrm{Pr}\{L_{\alpha_1}(t_1)\in\mathrm{d}x_1\}\mathrm{Pr}\{L_{\alpha_2}(t_2)\in\mathrm{d}x_2\}\\
 	&=\iint_{\mathbb{R}_+\times\mathbb{R}_+}e^{-u^2x_1x_2/2}\mathrm{Pr}\{L_{\alpha_1}(t_1)\in\mathrm{d}x_1\}\mathrm{Pr}\{L_{\alpha_2}(t_2)\in\mathrm{d}x_2\},\ u\in\mathbb{R}.
 \end{align*}
 Its double Laplace transform is 
 \begin{align*}
 	\iint_{\mathbb{R}_+\times\mathbb{R}_+}e^{-z_1t_1-z_2t_2}\mathbb{E}(iuW(L_{\alpha_1}(t_1)&,L_{\alpha_2}(t_2)))\,\mathrm{d}t_1\,\mathrm{d}t_2\\
 	&=z_1^{\alpha_1-1}z_2^{\alpha_2-1}\iint_{\mathbb{R}_+\times\mathbb{R}_+}e^{-u^2x_1x_2/2}e^{-z_1^{\alpha_1}x_1}e^{-z_2^{\alpha_2}x_2}\,\mathrm{d}x_1\,\mathrm{d}x_2\\
 	&=z_2^{\alpha_2-1}\int_{0}^{\infty}\frac{z_1^{\alpha_1-1}}{z_1^{\alpha_1}+(u^2x_2/2)}e^{-z_2^{\alpha_2}x_2}\,\mathrm{d}x_2,
 \end{align*}
 whose inversion with respect to $z_1$ yields
 \begin{align*}
 	\iint_{\mathbb{R}_+\times\mathbb{R}_+}e^{-z_2t_2}\mathbb{E}(iuW(L_{\alpha_1}(t_1),L_{\alpha_2}(t_2)))\,\mathrm{d}t_2&=z_2^{\alpha_2-1}\int_{0}^{\infty}E_{\alpha_1,1}(-u^2x_2t_1^{\alpha_1}/2)e^{-z_2^{\alpha_2}x_2}\,\mathrm{d}x_2\\
 	&=\sum_{k=0}^{\infty}\frac{(-u^2t_1^{\alpha_1}/2)^k}{\Gamma(\alpha_1 k+1)}z_2^{\alpha_2-1}\int_{0}^{\infty}x_2^ke^{-z_2^{\alpha_2}x_2}\,\mathrm{d}x_2\\
 	&=\sum_{k=0}^{\infty}\frac{(-u^2t_1^{\alpha_1}/2)^kk!}{\Gamma(\alpha_1 k+1)z_2^{\alpha_2 k+1}}.
 \end{align*}
 Further, on taking inverse Laplace transform with respect to $z_2$ gives (\ref{bstclap2}). Thus, for each $(t_1,t_2)\in\mathbb{R}^2_+$, $\tilde{\mathscr{X}}_{\alpha_1,\alpha_2}^{(n)}(t_1,t_2)$ converges in distribution to (\ref{bstc1}) as $n\to\infty$.
 \begin{remark}
 	For $\alpha_2=1$, (\ref{bstclap2}) reduces to (\ref{tcbslap}). Hence,
 	\begin{equation*}
 		\lim_{n\to\infty}\mathscr{X}^{(n)}_{\alpha,1}(t_1,t_2)\overset{d}{=}\lim_{n\to\infty}\tilde{\mathscr{X}}_{\alpha_1,\alpha_2}^{(n)}(t_1,t_2)\Big|_{\alpha_1=\alpha,\alpha_2=1},\ (t_1,t_2)\in\mathbb{R}^2_+,
 	\end{equation*}
 	where $\overset{d}{=}$ denotes the equality in distribution.
 \end{remark}

\subsection{Time fractional CPRF with exponential compounding} If the random variables $Y_1$, $Y_2$, $\dots$ in (\ref{fcprf}) are iid exponential random variables with parameter $\sigma>0$ then in (\ref{fcprfden}), we have $f_Y^{*n}=e^{-\sigma y}\sigma^ny^{n-1}/(n-1)!$, $n\ge1$. In (\ref{fcprfden}), by writing the distribution of FPRF in terms of the generalized wright function, the distribution function of time fractional CPRF is given by
		\begin{align}
			\mathrm{Pr}&\{\mathscr{Y}_{\alpha_1,\alpha_2}(t_1,t_2)\leq y\}\nonumber\\
			&={}_2\Psi_2\Bigg[\begin{matrix}
				(1,1),&(1,1)\\\\
				(1,\alpha_1),&(1,\alpha_2)
			\end{matrix}\bigg|-t_1^{\alpha_1}t_2^{\alpha_2}\lambda\Bigg]\textbf{1}_{[0,\infty)}(y)+\int_{-\infty}^{y}f_{\mathscr{Y}_{\alpha_1,\alpha_2}}(x,t_1,t_2)\,\mathrm{d}x,\ y\in\mathbb{R},\ (t_1,t_2)\in\mathbb{R}^2_+,\label{fcprfdf}
		\end{align} 
		where
		\begin{equation}\label{fcprfeden}
		f_{\mathscr{Y}_{\alpha_1,\alpha_2}}(x,t_1,t_2)=\frac{e^{-\sigma y}}{y}\sum_{n=1}^{\infty}\frac{(\lambda\sigma t_1^{\alpha_1}t_2^{\alpha_2}y)^n}{(n-1)!}{}_2\Psi_2\Bigg[\begin{matrix}
			(n+1,1),&(n+1,1)\\\\
			(n\alpha_1+1,\alpha_1),&(n\alpha_2+1,\alpha_2)
		\end{matrix}\bigg|-\lambda t_1^{\alpha_1}t_2^{\alpha_2}\Bigg]\textbf{1}_{[0,\infty)}(y).
		\end{equation}
	Here, ${}_2\Psi_2$ is the generalized Wright function defined in (\ref{genwrit}).
	\begin{theorem}
		For $\alpha_1=1$, the distribution function (\ref{fcprfdf}) reduces to
		\begin{equation*}
			\mathrm{Pr}\{\mathscr{Y}_{1,\alpha_2}(t_1,t_2)\leq y\}=E_{\alpha_2,1}(-\lambda t_1t_2^{\alpha_2})\textbf{1}_{[0,\infty)}(y)+\int_{-\infty}^{y}f_{\mathscr{Y}_{1,\alpha_2}}(x,t_1,t_2)\,\mathrm{d}x,\ y\in\mathbb{R},\ (t_1,t_2)\in\mathbb{R}^2_+.
		\end{equation*}
		Here,
		\begin{equation}\label{f}
			f_{\mathscr{Y}_{1,\alpha_2}}(x,t_1,t_2)=\frac{e^{-\sigma y}}{y}\sum_{n=1}^{\infty}\frac{(\lambda\sigma t_1t_2^{\alpha_2}y)^n}{(n-1)!}E_{\alpha_2,n\alpha_2+1}^{n+1}(-\lambda t_1t_2^{\alpha_2})\textbf{1}_{[0,\infty)}(y),
		\end{equation}
		where $E_{\alpha,\beta}^\gamma(\cdot)$, $\alpha>0$, $\beta\in\mathbb{R}$, $\gamma\in\mathbb{R}$ is the three parameter Mittag-Leffler function defined in (\ref{tpml}).
		
		The function $f_{\mathscr{Y}_{1,\alpha_2}}(x,t_1,t_2)$ solves the following partial differential equation:
		\begin{equation}\label{fdeq}
			\sigma\frac{\partial^{\alpha_2}}{\partial t_2^{\alpha_2}}f_{\mathscr{Y}_{1,\alpha_2}}(x,t_1,t_2)=-\bigg(\lambda t_1+\frac{\partial^{\alpha_2}}{\partial t_2^{\alpha_2}}\bigg)\frac{\partial}{\partial y}f_{\mathscr{Y}_{1,\alpha_2}}(x,t_1,t_2),\ t_1\ge0
		\end{equation}
		with $f_{\mathscr{Y}_{1,\alpha_2}}(x,t_1,0)=0$, where $\frac{\partial^{\alpha}}{\partial t_2^{\alpha}}$ is the Caputo fractional derivative defined in (\ref{caputoder}).
		 Also,
		\begin{equation*}
			\int_{0}^{\infty}f_{\mathscr{Y}_{1,\alpha_2}}(x,t_1,t_2)\,\mathrm{d}y=1-E_{\alpha_2,1}(-\lambda t_1t_2^{\alpha_2}).
		\end{equation*}
	\end{theorem}
	\begin{proof}
		For $\alpha_1=1$, the generalized Wright function in (\ref{fcprfeden}) reduces to $E_{\alpha_2,n\alpha_2+1}^{n+1}(-\lambda t_1t_2^{\alpha_2})$. Also, the differential equation (\ref{fdeq}) can be obtain by taking the appropriate derivatives of (\ref{f}). This completes the proof.
	\end{proof}
\begin{remark}
	For $\alpha_1=\alpha_2=1$, the time fractional CPRF reduces to the CPRF whose distribution function is given by
	\begin{equation*}
		\mathrm{Pr}\{\mathscr{Y}(t_1,t_2)\leq y\}=e^{-\lambda t_1t_2}\textbf{1}_{[0,\infty)}(y)+\int_{-\infty}^{y}f_{\mathscr{Y}}(x,t_1,t_2)\,\mathrm{d}x,
	\end{equation*}
	where 
	\begin{equation*}
		f_{\mathscr{Y}}(y,t_1,t_2)=\frac{e^{-\sigma y-\lambda t_1t_2}}{y}\sum_{n=1}^{\infty}\frac{(\lambda\sigma t_1t_2y)^n}{(n-1)!n!}\textbf{1}_{[0,\infty)}(y).
	\end{equation*}
	It solves
	\begin{equation*}
		\sigma\frac{\partial}{\partial t_1}f_{\mathscr{Y}}(y,t_1,t_2)=-\bigg(\lambda t_2+\frac{\partial}{\partial t_1}\bigg)\frac{\partial}{\partial y}f_{\mathscr{Y}}(y,t_1,t_2)
	\end{equation*}
	and
	\begin{equation*}
	\sigma\frac{\partial}{\partial t_2}f_{\mathscr{Y}}(y,t_1,t_2)=-\bigg(\lambda t_1+\frac{\partial}{\partial t_2}\bigg)\frac{\partial}{\partial y}f_{\mathscr{Y}}(y,t_1,t_2)
	\end{equation*}
	with initial conditions $f_{\mathscr{Y}}(y,0,t_2)=0$ and $f_{\mathscr{Y}}(y,t_1,0)=0$, respectively. Also, we have
	\begin{equation*}
		\int_{0}^{\infty}f_{\mathscr{Y}}(y,t_1,t_2)\,\mathrm{d}y=1-e^{-\lambda t_1t_2},\ (t_1,t_2)\in\mathbb{R}^2_+.
	\end{equation*}
\end{remark}

\subsection{Time fractional CPRF with Mittag-Leffler compounding} For $0<\beta\leq1$, let $Y_1^\beta$, $ Y_2^\beta$, $\dots$ be iid random variable with common density function
\begin{equation*}
f_{Y^\beta}(y)=\sigma y^{\beta-1}E_{\beta,\beta}(-\sigma y^\beta)\textbf{1}_{[0,\infty)}(y),\ \sigma>0,
\end{equation*} 
where $E_{\alpha,\beta}(\cdot)$, $\alpha>0$, $\beta\in\mathbb{R}$ is  the two parameter Mittag-Leffler function defined in (\ref{tpml}). So, the density of $\sum_{j=1}^{n}Y_j^\beta$ is given by
\begin{equation}\label{sf1fold}
	f^{*n}_{Y^\beta}(y)=\sigma^ny^{\beta n-1}E_{\beta,\beta n}^n(-\sigma y^\alpha)\textbf{1}_{[0,\infty)}(y),
\end{equation}
whose Laplace transform is
\begin{equation}\label{mldistflap}
	\int_{0}^{\infty}e^{-uy}f^{*n}_{Y^\beta}(y)\,\mathrm{d}y=\frac{\sigma^n}{(u^\beta+\sigma)^n},\ u>0.
\end{equation}

Let us consider a time fractional CPRF defined as follows:
\begin{equation}\label{tfcprfml}
	\mathscr{Y}_{\alpha_1,\alpha_2}^\beta(t_1,t_2)\coloneqq\sum_{j=1}^{N_{\alpha_1,\alpha_2}(t_1,t_2)}Y_j^\beta,\ (t_1,t_2)\in\mathbb{R}^2_+,
\end{equation}
where the FPRF $\{N_{\alpha_1,\alpha_2}(t_1,t_2),\ (t_1,t_2)\in\mathbb{R}^2_+\}$ is independent of $\{Y_j^\beta\}_{j\ge1}$.
\begin{proposition}
	Let $\{H_\beta(t),\ t\ge0\}$ be one parameter stable subordinator as defined in Section \ref{sub} which is independent of the time fractional CPRF $\{\mathscr{Y}_{\alpha_1,\alpha_2}(t_1,t_2),\ (t_1,t_2)\in\mathbb{R}^2_+\}$ with exponential compounding. Then, the following time-changed relationship holds:
	\begin{equation*}
		\mathscr{Y}_{\alpha_1,\alpha_2}^\beta(t_1,t_2)\overset{d}{=}H_\beta(\mathscr{Y}_{\alpha_1,\alpha_2}(t_1,t_2)).
	\end{equation*}
\end{proposition}
\begin{proof}	
	From (\ref{sublap}), it follows that the Laplace transform of $H_\beta(\mathscr{Y}_{\alpha_1,\alpha_2}(t_1,t_2))$ is given by
	\begin{align*}
		\mathbb{E}\exp(-uH_\beta(\mathscr{Y}(t_1,t_2)))&=\int_{0}^{\infty}\mathbb{E}e^{-uH_\beta(x)}\mathrm{Pr}\{\mathscr{Y}_{\alpha_1,\alpha_2}(t_1,t_2)\in\mathrm{d}x\}\\
		&=\int_{0}^{\infty}e^{-xu^\beta}\mathrm{Pr}\{\mathscr{Y}_{\alpha_1,\alpha_2}(t_1,t_2)\in\mathrm{d}x\}\\
		&=\mathbb{E}\bigg(\frac{\sigma}{u^\beta+\sigma}\bigg)^{N_{\alpha_1,\alpha_2}(t_1,t_2)}\\
		&=\mathbb{E}\exp(-u\mathscr{Y}_{\alpha_1,\alpha_2}^\beta(t_1,t_2))
	\end{align*}
	where the last step follows from (\ref{mldistflap}). The uniqueness of Laplace transform completes the proof.
\end{proof}

The following result follows from (\ref{sf1fold}) and by taking the appropriate fractional derivatives.
\begin{theorem}
	For $\alpha_2=1$, the distribution function of (\ref{tfcprfml}) is given by
	\begin{equation*}
		\mathrm{Pr}\{\mathscr{Y}_{\alpha_1,\alpha_2}^\beta(t_1,t_2)\leq y\}=E_{\alpha_1,1}(-\lambda t_1^{\alpha_1}t_2)\textbf{1}_{[0,\infty)}(y)+\int_{-\infty}^{y}f_{\mathscr{Y}_{\alpha_1,1}^\beta}(x,t_1,t_2)\,\mathrm{d}x,\ y\ge0,\ (t_1,t_2)\in\mathbb{R}^2_+,
	\end{equation*}
	where
	\begin{equation*}
	f_{\mathscr{Y}_{\alpha_1,1}^\beta}(y,t_1,t_2)=\frac{1}{y}\sum_{n=1}^{\infty}(\lambda\sigma t_1^{\alpha_1}t_2y^\beta)^nE_{\alpha_1,\alpha_1 n+1}^{n+1}(-\lambda t_1^{\alpha_1}t_2)E_{\beta,\beta n}^n(-\sigma y^\beta).
	\end{equation*}
	It solves
	\begin{equation*}
		\sigma\frac{\partial^{\alpha_1}}{\partial t_1^{\alpha_1}}f_{\mathscr{Y}_{\alpha_1,1}^\beta}(y,t_1,t_2)=-\bigg(\lambda t_1+\frac{\partial^{\alpha_1}}{\partial t_1^{\alpha_1}}\bigg)\frac{\partial^\beta}{\partial y^\beta}f_{\mathscr{Y}_{\alpha_1,1}^\beta}(y,t_1,t_2),
	\end{equation*}
	with $f_{\mathscr{Y}_{\alpha_1,1}^\beta}(y,0,t_2)=0$. Also,
	\begin{equation*}
		\int_{0}^{\infty}f_{\mathscr{Y}_{\alpha_1,1}^\beta}(y,t_1,t_2)\,\mathrm{d}y=1-E_{\alpha_1,1}(-\lambda t_1^{\alpha_1}t_2).
	\end{equation*}
\end{theorem}

\begin{remark}
	Note that for $\alpha_1=\alpha_2=1$, the random field (\ref{tfcprfml}) reduces to the CPRF with Mittag-Leffler compounding variables
	\begin{equation*}
		\mathscr{Y}^\beta(t_1,t_2)\coloneqq\sum_{j=1}^{N(t_1,t_2)}Y_j^\beta,\ (t_1,t_2)\in\mathbb{R}^2_+.
	\end{equation*} 
	Its distribution function is given by
	\begin{equation*}
		\mathrm{Pr}\{\mathscr{Y}_\beta(t_1,t_2)\leq y\}=e^{-\lambda t_1t_2}\textbf{1}_{[0,\infty)}(y)+\int_{-\infty}^{y}f_{\mathscr{Y}_\beta}(x,t_1,t_2)\,\mathrm{d}x,\ y\in\mathbb{R},\ (t_1,t_2)\in\mathbb{R}^2_+.
	\end{equation*}
	Here,
	\begin{equation}\label{sf1den}
	f_{\mathscr{Y}_\beta}(y,t_1,t_2)=\frac{e^{-\lambda t_1t_2}}{y}\sum_{n=1}^{\infty}\frac{(\sigma\lambda t_1t_2 y^\beta)^n}{n!}E_{\beta,\beta n}^n(-\sigma y^\beta)\textbf{1}_{[0,\infty)}(y),
	\end{equation}
	which solves
	\begin{equation*}
		\sigma\frac{\partial}{\partial t_1}f_{\mathscr{Y}_\beta}(y,t_1,t_2)=-\bigg(\lambda t_2+\frac{\partial}{\partial t_1}\bigg)\frac{\partial^\beta}{\partial y^\beta}f_{\mathscr{Y}_\beta}(y,t_1,t_2)
	\end{equation*}
	and
	\begin{equation*}
		\sigma\frac{\partial}{\partial t_2}f_{\mathscr{Y}_\beta}(y,t_1,t_2)=-\bigg(\lambda t_1+\frac{\partial}{\partial t_2}\bigg)\frac{\partial^\beta}{\partial y^\beta}f_{\mathscr{Y}_\beta}(y,t_1,t_2),
	\end{equation*}
	with initial conditions $f_{\mathscr{Y}_\beta}(y,0,t_2)=0$ and $f_{\mathscr{Y}_\beta}(y,t_1,0)=0$, respectively. Also,
	\begin{equation*}
		\int_{0}^{\infty}f_{\mathscr{Y}_\beta}(y,t_1,t_2)\,\mathrm{d}y=1-e^{-\lambda t_1t_2}.
	\end{equation*}
	Moreover, we have the following time-changed relation:
	\begin{equation*}
		\mathscr{Y}_\beta(t_1,t_2)\overset{d}{=}H_\beta(\mathscr{Y}(t_1,t_2)),
	\end{equation*}
	where $\{H_\beta(t),\ t\ge0\}$ is the one parameter stable subordinator which is independent of the CPRF $\{\mathscr{Y}(t_1,t_2),\ (t_1,t_2)\in\mathbb{R}^2_+\}$ with exponential compounding.
\end{remark}

\section{Some fractional variants of PRF}\label{sec5}
 Here, we introduce and study  some fractional variants of the Poisson random field (PRF). First, we define a space fractional Poisson random field. 
 
 From (\ref{prfdist}), we note that
 \begin{equation*}
 	\lambda^2\frac{\partial}{\partial\lambda}p(n,t_1,t_2)=n\lambda p(n,t_1,t_2)-(n+1)\lambda p(n+1,t_1,t_2),\ n\ge0.
 \end{equation*} 
 So, the governing equations (\ref{prfeqs}) can be rewritten as follows:
\begin{align}\label{prfeqs1}
	\frac{\partial^2}{\partial t_2\partial t_1}p(n,t_1,t_2)
	&=(n+1)\lambda p(n+1,t_1,t_2)-n\lambda p(n,t_1,t_2)-\lambda p(n,t_1,t_2)+\lambda p(n-1,t_1,t_2)\nonumber\\
	&\ \ +(n-1)\lambda p(n-1,t_1,t_2)-n\lambda p(n,t_1,t_2)\nonumber\\
	&=-\lambda^2\frac{\partial}{\partial\lambda}p(n,t_1,t_2)-\lambda p(n,t_1,t_2)+\lambda p(n-1,t_1,t_2)+\lambda^2\frac{\partial}{\partial\lambda}p(n-1,t_1,t_2)\nonumber\\
	&=-\bigg(\lambda^2\frac{\partial}{\partial\lambda}+\lambda\bigg)(I-B)p(n,t_1,t_2),\ n\ge0,
\end{align}
where $B$ denotes the backward shift operator, that is, $Bp(n,t_1,t_2)=p(n-1,t_1,t_2)$.
\subsection{Space fractional PRF} Let us consider a random field $\{N_\alpha(t_1,t_2),\ (t_1,t_2)\in\mathbb{R}^2_+\}$, $0<\alpha\leq1$ whose state probabilities $p_\alpha(n,t_1,t_2)=\mathrm{Pr}\{N_\alpha(t_1,t_2)=n\}$, $n\ge0$ solve
\begin{equation}\label{sfprfequs}
\frac{\partial^2}{\partial t_2\partial t_1}p_\alpha(n,t_1,t_2)=-\bigg(\frac{\lambda^{\alpha+1}}{\alpha}\frac{\partial}{\partial\lambda}+\lambda^\alpha\bigg)(I-B)^\alpha p_\alpha(n,t_1,t_2),\ n\ge0,
\end{equation}
with $p_\alpha(0,0,t_2)=p_\alpha(0,t_1,0)=p_\alpha(0,0,0)=1$ for all $t_1\ge0$ and $t_2\ge0$, and $p(n,t_1,t_2)=0$ for all $n<0$. Here, 
$
(I-B)^\alpha=\sum_{k=0}^{\infty}{(\alpha)_k}(-B)^k/{k!}.
$

We call the process $\{N_\alpha(t_1,t_2),\ (t_1,t_2)\in\mathbb{R}^2_+\}$ as the space fractional Poisson random field (SFPRF). For $\alpha=1$, the system of differential equations (\ref{sfprfequs}) reduces to (\ref{prfeqs1}). Thus, $N_1(t_1,t_2)\overset{d}{=}N(t_1,t_2)$, where $\overset{d}{=}$ denotes the equality in distribution.

In the following result, we show that the SFPRF is equal is distribution to a one parameter Poisson process time-changed by an independent two parameter stable subordinator.
\begin{theorem}
	Let $\{N(t),\ t\ge0\}$ be a Poisson process with transition rate $\lambda>0$ which is independent of a two parameter stable subordinator $\{H_\alpha(t_1,t_2),\ (t_1,t_2)\in\mathbb{R}^2_+\}$, $0<\alpha<1$, defined in Section \ref{tpsub}. Then, the SFPRF has the following time-changed representation: 
	\begin{equation}\label{spt2}
		N_\alpha(t_1,t_2)\overset{d}{=} N(H_\alpha(t_1,t_2)),\ (t_1,t_2)\in\mathbb{R}^2_+.
	\end{equation}
\end{theorem}
\begin{proof}
	The pgf  $G_\alpha(u,t_1,t_2)=\mathbb{E}u^{N(H_\alpha(t_1,t_2))}$, $0<u\leq1$ of right hand side of (\ref{tpsublap}) is given by
	\begin{equation}\label{spt2pgf}
		G_\alpha(u,t_1,t_2)=\int_{0}^{\infty}e^{\lambda x(u-1)}\mathrm{Pr}\{S_\alpha(t_1,t_2)\in\mathrm{d}x\}=\exp\big(-t_1t_2\lambda^\alpha(1-u)^\alpha\big),
	\end{equation}
	where we have used the pgf of Poisson process, that is,  $\mathbb{E}u^{N(t)}=e^{\lambda t(u-1)}$.	
	
	On taking the derivative of (\ref{spt2pgf}) with respect to $t_1$, we get
	\begin{equation*}
		\frac{\partial}{\partial t_1}G_\alpha(u,t_1,t_2)=-t_2\lambda^\alpha(1-u)^\alpha e^{-t_1t_2\lambda^\alpha(1-u)^\alpha},
	\end{equation*}
	whose derivative with respect to $t_2$ yields
	\begin{align}
		\frac{\partial^2}{\partial t_2\partial t_1}G_\alpha(u,t_1,t_2)&=t_1t_2(\lambda^\alpha(1-u)^\alpha)^2e^{-t_1t_2\lambda^\alpha(1-u)^\alpha}-\lambda^\alpha(1-u)^\alpha e^{-t_1t_2\lambda^\alpha(1-u)^\alpha}\nonumber\\
		&=-\frac{\lambda^{\alpha+1}(1-u)^{\alpha}}{\alpha}\frac{\partial}{\partial \lambda}e^{-t_1t_2\lambda^\alpha(1-u)^\alpha}-\lambda^\alpha(1-u)^\alpha e^{-t_1t_2\lambda^\alpha(1-u)^\alpha}\nonumber\\
		&=-\frac{\lambda^{\alpha+1}(1-u)^{\alpha}}{\alpha}\frac{\partial}{\partial \lambda}G_\alpha(u,t_1,t_2)-\lambda^\alpha(1-u)^\alpha G_\alpha(u,t_1,t_2).\label{sfprfpgfeq}
\end{align}
On substituting $G_\alpha(u,t_1,t_2)=\sum_{n=0}^{\infty}u^np_\alpha(n,t_1,t_2)$ in (\ref{sfprfpgfeq}), we obtain
\begin{align}
			\sum_{n=0}^{\infty}u^n\frac{\partial^2}{\partial t_2\partial t_1}p_\alpha(n,t_1,t_2)&=-\bigg(\frac{\lambda^{\alpha+1}}{\alpha}\frac{\partial}{\partial\lambda}+\lambda^\alpha\bigg)\sum_{k=0}^{\infty}\frac{(\alpha)_k}{k!}(-u)^k\sum_{n=0}^{\infty}u^np_\alpha(n,t_1,t_2)\nonumber\\
		&=-\bigg(\frac{\lambda^{\alpha+1}}{\alpha}\frac{\partial}{\partial\lambda}+\lambda^\alpha\bigg)\sum_{k=0}^{\infty}(-1)^k\frac{(\alpha)_k}{k!}\sum_{n=k}^{\infty}u^np_\alpha(n-k,t_1,t_2)\nonumber\\
		&=-\bigg(\frac{\lambda^{\alpha+1}}{\alpha}\frac{\partial}{\partial\lambda}+\lambda^\alpha\bigg)\sum_{n=0}^{\infty}u^n\sum_{k=0}^{n}(-1)^k\frac{(\alpha)_k}{k!}p_\alpha(n-k,t_1,t_2).\label{thm3pf1}
	\end{align}
	On comparing the coefficient of $u^n$ on both sides of (\ref{thm3pf1}), we get 
	\begin{align*}
			\frac{\partial^2}{\partial t_2\partial t_1}p_\alpha(n,t_1,t_2)&=-\bigg(\frac{\lambda^{\alpha+1}}{\alpha}\frac{\partial}{\partial\lambda}+\lambda^\alpha\bigg)\sum_{k=0}^{n}(-1)^k\frac{(\alpha)_k}{k!}p_\alpha(n-k,t_1,t_2)\\
			&=-\bigg(\frac{\lambda^{\alpha+1}}{\alpha}\frac{\partial}{\partial\lambda}+\lambda^\alpha\bigg)\sum_{k=0}^{\infty}\frac{(\alpha)_k}{k!}(-B)^kp_\alpha(n,t_1,t_2)\\
			&=-\bigg(\frac{\lambda^{\alpha+1}}{\alpha}\frac{\partial}{\partial\lambda}+\lambda^\alpha\bigg)(I-B)^\alpha p_\alpha(n,t_1,t_2),
	\end{align*}
	which coincide with (\ref{sfprfequs}). This completes the proof.
\end{proof}
\begin{remark}
The pgf (\ref{spt2pgf}) can be rewritten as follows:	
	\begin{equation}\label{thm3pf0}
		G_\alpha(u,t_1,t_2)=\sum_{k=0}^{\infty}\frac{(-t_1t_2\lambda^\alpha)^k}{k!}(1-u)^{\alpha k}=\sum_{k=0}^{\infty}\frac{(-t_1t_2\lambda^\alpha)^k}{k!}\sum_{n=0}^{\infty}\frac{(\alpha k)_n}{n!}(-u)^n.
	\end{equation}
	In view of (\ref{spt2}), the coefficient of $u^n$ in (\ref{thm3pf0}) is the distribution of SFPRF. It is given by
	\begin{equation}\label{spt2dist}
		p_\alpha(n,t_1,t_2)=\sum_{k=0}^{\infty}(-1)^{n+k}\frac{(\alpha k)_n(t_1t_2\lambda^\alpha)^k}{n!k!},\ n\ge0,
	\end{equation}
	where $(\alpha k)_n=\alpha k(\alpha k-1)\dots(\alpha k-n+1)=\Gamma(\alpha k+1)/\Gamma(\alpha k+1-n)$.

In particular, for $\alpha=1$, the distribution (\ref{spt2dist}) reduces to 
	\begin{align*}
		p_1(n,t_1,t_2)&=\sum_{k=0}^{\infty}(-1)^{n+k}\frac{(k)_n(t_1t_2\lambda)^k}{n!k!}\\
		&=\sum_{k=n}^{\infty}(-1)^{n+k}\frac{(t_1t_2\lambda)^k}{(k-n)!n!}=\frac{(t_1t_2\lambda)^n}{n!}e^{-t_1t_2\lambda},\ n\ge0,
	\end{align*}
	which agrees with the distribution of PRF given (\ref{prfdist}).
\end{remark}

Next, we derive the governing equation for the density of two parameter stable subordinator.
\begin{proposition}
	The density $g_\alpha(x,t_1,t_2)=\mathrm{Pr}\{H_\alpha(t_1,t_2)\in\mathrm{d}x\}/\mathrm{d}x$, $x\ge0$, $(t_1,t_2)\in\mathbb{R}^2_+$ solves the following fractional differential equation:
	\begin{equation}\label{subdeneq}
		\frac{\partial}{\partial t_1}g_\alpha(x,t_1,t_2)=-t_2\frac{\partial^\alpha}{\partial x^\alpha}g_\alpha(x,t_1,t_2),
	\end{equation}
	with $g_\alpha(0,t_1,t_2)=0$ and $g_\alpha(x,0,t_2)=\delta(x)$. Here, $\frac{\partial^\alpha}{\partial x^\alpha}$ is the Caputo fractional derivative and $\delta(x)$ denotes the Dirac delta function.
\end{proposition}
\begin{proof}
	In view of (\ref{caputoderlap}) and on taking the Laplace transform $\tilde{g}_\alpha(u,t_1,t_2)=\int_{0}^{\infty}e^{-ux}g_\alpha(x,t_1,t_2)\,\mathrm{d}x$ on both sides of (\ref{subdeneq}), we get
	\begin{equation}\label{pf1}
		\frac{\partial}{\partial t_1}\tilde{g}_\alpha(u,t_1,t_2)=-t_2u^{\alpha}\tilde{g}_\alpha(u,t_1,t_2)
	\end{equation}
	with $\tilde{g}_\alpha(u,0,t_2)=1$. On solving (\ref{pf1}), we obtain $\tilde{g}_\alpha(u,t_1,t_2)=e^{-t_1t_2u^\alpha}$, $u>0$. This completes the proof.
\end{proof}
\subsection{Space-time fractional PRF} For $\alpha_i\in(0,1]$, $i=1,2$ and $0<\beta\leq1$, let us consider a random field $\{N_{\alpha_1,\alpha_2}^\beta(t_1,t_2),\ (t_1,t_2)\in\mathbb{R}^2_+\}$ whose distribution $p_{\alpha_1,\alpha_2}^\beta(n,t_1,t_2)=\mathrm{Pr}\{N_{\alpha_1,\alpha_2}^\beta(t_1,t_2)=n\}$ solves the following system of fractional differential equations:
\begin{equation}\label{stf}
\frac{\partial^{\alpha_1+\alpha_2}}{\partial t_2^{\alpha_2}\partial t_1^{\alpha_1}}p_{\alpha_1,\alpha_2}^\beta(n,t_1,t_2)=-\bigg(\frac{\lambda^{\beta+1}}{\beta}\frac{\partial}{\partial\lambda}+\lambda^\beta\bigg)(I-B)^\beta p_{\alpha_1,\alpha_2}^\beta(n,t_1,t_2),\ n\ge0,
\end{equation}
with $p_{\alpha_1,\alpha_2}^\beta(0,0,t_2)=p_{\alpha_1,\alpha_2}^\beta(0,t_1,0)=p_{\alpha_1,\alpha_2}^\beta(0,0,0)=1$. In (\ref{stf}), the derivative involve is the Caputo fractional derivative defined in (\ref{caputoder}). We call it the space-time fractional Poisson random field (STFPRF). For $\beta=1$, it reduces to the time fractional Poisson random field (TFPRF) introduced and studied in \cite{Kataria2024}. Also, for $\alpha_1=\alpha_2=1$, it reduces to the SFPRF.

The pgf $G_{\alpha_1,\alpha_2}^\beta(u,t_1,t_2)=\sum_{n=0}^{\infty}u^np_{\alpha_1,\alpha_2}^\beta(n,t_1,t_2)$ solves the following fractional differential equation:
\begin{equation}
	\frac{\partial^{\alpha_1+\alpha_2}}{\partial t_2^{\alpha_2}\partial t_1^{\alpha_1}}G_{\alpha_1,\alpha_2}^\beta(u,t_1,t_2)=-\bigg(\frac{\lambda^{\beta+1}}{\beta}\frac{\partial}{\partial \lambda}+\lambda^\beta\bigg)(1-u)^\beta G_{\alpha_1,\alpha_2}^\beta(u,t_1,t_2)
\end{equation}
with $G_{\alpha_1,\alpha_2}^\beta(u,0,t_2)=G_{\alpha_1,\alpha_2}^\beta(u,t_1,0)=1$.
\begin{theorem}\label{thm5}
	Let $\{T_{2\alpha_i}(t),\ t\ge0\}$, $0<\alpha_i\leq1$, $i=1,2$ 
	be two random processes whose corresponding densities are folded solutions of the following Cauchy problems:
	\begin{equation}\label{couchyp}
		\frac{\partial^{2\alpha_i}}{\partial t^{2\alpha_i}}u_{2\alpha_i}(x,t)=\frac{\partial^2}{\partial x^2}u_{2\alpha_i}(x,t),\ t>0,\,x\in\mathbb{R},
	\end{equation}	
	with initial condition $u_{2\alpha_i}(x,0)=\delta(x),\ 0<\alpha_i\leq1$, and $(\partial/\partial t)u_{2\alpha_i}(x,t)|_{t=0}=0,\ \frac{1}{2}<\alpha_i\leq1$.	
	Also, let $\{H_\beta(t_1,t_2),\ (t_1,t_2)\in\mathbb{R}^2_+\}$ be a two parameter stable subordinator. Then, the STFPRF satisfies the following time-change relationships: 
	\begin{equation}\label{sfprfrep}
		N_{\alpha_1,\alpha_2}^\beta(t_1,t_2)\overset{d}{=}N_\beta(T_{2\alpha_1}(t_1),T_{2\alpha_2}(t_2))\overset{d}{=}N(H_\beta(T_{2\alpha_1}(t_1),T_{2\alpha_2}(t_2))),\ (t_1,t_2)\in\mathbb{R}^2_+,
	\end{equation}
	where $\{N(t),\ t\ge0\}$ and $\{N_\beta(t_1,t_2),\ (t_1,t_2)\in\mathbb{R}^2_+\}$ are the Poisson process and the SFPRF, respectively. In (\ref{sfprfrep}), the component processes involved are independent of each other.
\end{theorem}
\begin{proof}
The pgf of right hand side of (\ref{sfprfrep}) is given by
\begin{equation*}
	\mathbb{E}u^{N_\beta(T_{2\alpha_1}(t_1),T_{2\alpha_2}(t_2))}=\iint_{\mathbb{R}_+\times\mathbb{R}_+}G_\beta(u,x_1,x_2)\mathrm{Pr}\{T_{2\alpha_1}(t_1)\in\mathrm{d}x_1\}\mathrm{Pr}\{T_{2\alpha_2}(t_2)\in\mathrm{d}x_2\},
\end{equation*}
where $G_\beta(u,x_1,x_2)$ is the pgf of SFPRF. Its double Laplace transform with respect to $t_1$ and $t_2$ is given by
\begin{equation}\label{thm4pf0}
	\mathcal{L}(u,z_1,z_2)=z_1^{\alpha_1-1}z_2^{\alpha_2-1}\iint_{\mathbb{R}_+\times\mathbb{R}_+}G_\beta(u,x_1,x_2)e^{-z_1^{\alpha_1}x_1}e^{-z_2^{\alpha_2}x_2}\,\mathrm{d}x_1\,\mathrm{d}x_2,\ z_i>0,\ i=1,2,
\end{equation}
where we have used the following result (see \cite{Orsingher2004}, Eq. (3.3)):
\begin{equation}\label{tclap}
	\int_{0}^{\infty}e^{-z_it_i}\mathrm{Pr}(T_{2\alpha_i}(t_i)\in\mathrm{d}x_i)\,\mathrm{d}t_i=z_i^{\alpha_i-1}e^{-z_i^{\alpha_i}x_i}\,\mathrm{d}x_i,\ i=1,2.
\end{equation}

For $\alpha=\beta$, on taking the Laplace transform on both sides of (\ref{sfprfpgfeq}) with respect to $t_1$, we get
\begin{equation*}
	z_1\frac{\partial}{\partial t_2}\int_{0}^{\infty}e^{-z_1t_2}G_\beta(u,t_1,t_2)\,\mathrm{d}t_1=
	-\bigg(\frac{\lambda^{\beta+1}}{\beta}\frac{\partial}{\partial \lambda}+\lambda^\beta\bigg)\int_{0}^{\infty}e^{-z_1t_1}G_\beta(u,t_1,t_2)\,\mathrm{d}t_1,
\end{equation*}
where we have used $G_\beta(u,0,t_2)=1$. Its Laplace transform with respect to $t_2$ yields
\begin{align*}
	z_1z_2\iint_{\mathbb{R}_+\times\mathbb{R}_+}e^{-z_1t_1-z_2t_2}&G_\beta(u,t_1,t_2)\,\mathrm{d}t_1\,\mathrm{d}t_2-1\nonumber\\
	&=-\bigg(\frac{\lambda^{\beta+1}}{\beta}\frac{\partial}{\partial \lambda}+\lambda^\beta\bigg)\iint_{\mathbb{R}_+\times\mathbb{R}_+}e^{-z_1t_1-z_2t_2}G_\beta(u,t_1,t_2)\,\mathrm{d}t_1\,\mathrm{d}t_1,
\end{align*}
where we have used $\int_{0}^{\infty}e^{-z_1t_1}G_\beta(u,t_1,0)\,\mathrm{d}t_1=1/z_1$. So,
\begin{align}
	z_1^{\alpha_1-1}z_2^{\alpha_2-1}&z_1^{\alpha_1}z_2^{\alpha_2}\iint_{\mathbb{R}_+\times\mathbb{R}_+}e^{-z_1^{\alpha_1}t_1-z_2^{\alpha_2}t_2}G_\beta(u,t_1,t_2)\,\mathrm{d}t_1\,\mathrm{d}t_2-z_1^{\alpha_1-1}z_2^{\alpha_2-1}\nonumber\\
	&=-\bigg(\frac{\lambda^{\beta+1}}{\beta}\frac{\partial}{\partial \lambda}+\lambda^\beta\bigg)z_1^{\alpha_1-1}z_2^{\alpha_2-1}\iint_{\mathbb{R}_+\times\mathbb{R}_+}e^{-z_1^{\alpha_1}t_1-z_2^{\alpha_2}t_2}G_\beta(u,t_1,t_2)\,\mathrm{d}t_1\,\mathrm{d}t_2.\label{thm4pf1}
\end{align}
On using (\ref{thm4pf0}) in (\ref{thm4pf1}), we have
\begin{equation}\label{thm4pf3}
	z_1^{\alpha_1}z_2^{\alpha_2}\mathcal{L}(u,z_1,z_2)-z_1^{\alpha_1-1}z_2^{\alpha_2-1}=-\bigg(\frac{\lambda^{\beta+1}}{\beta}\frac{\partial}{\partial \lambda}+\lambda^\beta\bigg)\mathcal{L}(u,z_1,z_2).
\end{equation}

Now, on taking the double Laplace transform with respect to $t_1$ and $t_2$ on both sides of (\ref{sfprfpgfeq}), we get
\begin{align}
	z_1^{\alpha_1}z_2^{\alpha_2}\iint_{\mathbb{R}_+\times\mathbb{R}_+}&e^{-z_1t_1-z_2t_2}G_{\alpha_1,\alpha_2}^\beta(u,t_1,t_2)\,\mathrm{d}t_1\,\mathrm{d}t_2-z_1^{\alpha_1-1}z_2^{\alpha_2-1}\nonumber\\
	&=-\bigg(\frac{\lambda^{\beta+1}}{\beta}\frac{\partial}{\partial \lambda}+\lambda^\beta\bigg)\iint_{\mathbb{R}_+\times\mathbb{R}_+}e^{-z_1t_1-z_2t_2}G_{\alpha_1,\alpha_2}^\beta(u,t_1,t_2)\,\mathrm{d}t_1\,\mathrm{d}t_2.\label{thm4pf4}
\end{align}
From (\ref{thm4pf3}) and (\ref{thm4pf4}), it follows that $\iint_{\mathbb{R}_+\times\mathbb{R}_+}e^{-z_1t_1-z_2t_2}G_{\alpha_1,\alpha_2}^\beta(u,t_1,t_2)\,\mathrm{d}t_1\,\mathrm{d}t_2$ solves (\ref{thm4pf3}). Thus, the proof of first equality in distribution of (\ref{sfprfrep}) follows from the uniqueness of pgf. Also, the second equality in distribution is a consequence of (\ref{spt2}). This completes the proof.
\end{proof}
\begin{remark}\label{rem5}
	The Laplace transform of the composition $H_\beta(T_{2\alpha_1}(t_1),T_{2\alpha_2}(t_2))$ is given by
	\begin{equation*}
		\mathbb{E}\exp(-uH_\beta(T_{2\alpha_1}(t_1),T_{2\alpha_2}(t_2)))=\iint_{\mathbb{R}_+\times\mathbb{R}_+}e^{-x_1x_2u^\beta}\mathrm{Pr}\{T_{2\alpha_1}(t_1)\in\mathrm{d}x_1\}\mathrm{Pr}\{T_{2\alpha_2}(t_2)\in\mathrm{d}x_2\}.
	\end{equation*}
	By using (\ref{tclap}), its double Laplace transform is 
	\begin{align*}
	\iint_{\mathbb{R}_+\times\mathbb{R}_+}e^{-z_1t_1-z_2t_2}\mathbb{E}\exp(-uH_\beta(T_{2\alpha_1}&(t_1),T_{2\alpha_2}(t_2)))\,\mathrm{d}t_1\,\mathrm{d}t_2\\
	&=z_1^{\alpha_1-1}z_2^{\alpha_2-1}\iint_{\mathbb{R}_+\times\mathbb{R}_+}e^{-x_1x_2u^\beta}e^{-z_1^{\alpha_1}x_1}e^{-z_2^{\alpha_2}x_2}\,\mathrm{d}x_1\,\mathrm{d}x_2\\
	&=z_2^{\alpha_2-1}\int_{0}^{\infty}\frac{z_1^{\alpha_1-1}}{z_1^{\alpha_1}+x_2u^\beta}e^{-z_2^{\alpha_2}x_2}\,\mathrm{d}x_2,\ z_1>0,\, z_2>0,
	\end{align*}
	whose inversion with respect to $z_1$ yields
	\begin{align*}
		\int_{0}^{\infty}e^{-z_2t_2}\mathbb{E}\exp(-uH_\beta(T_{2\alpha_1}(t_1),T_{2\alpha_2}(t_2)))\,\mathrm{d}t_2&=z_2^{\alpha_2-1}\int_{0}^{\infty}E_{\alpha_1,1}(-x_2u^\beta t_1^{\alpha_1})e^{-z_2^{\alpha_2}x_2}\,\mathrm{d}x_2\\
		&=z_2^{\alpha_2-1}\sum_{k=0}^{\infty}\frac{(-u^\beta t_1^{\alpha_1})^k}{\Gamma(k\alpha_1+1)}\int_{0}^{\infty}x_2^ke^{-z_2^{\alpha_2}x_2}\,\mathrm{d}x_2\\
		&=\sum_{k=0}^{\infty}\frac{(-u^\beta t_1^{\alpha_1})^k}{\Gamma(k\alpha_1+1)}\frac{\Gamma(k+1)}{z_2^{\alpha_2 k+1}}.
	\end{align*}
	Hence,
	\begin{equation*}
		\mathbb{E}\exp(-uH_\beta(T_{2\alpha_1}(t_1),T_{2\alpha_2}(t_2)))=\sum_{k=0}^{\infty}\frac{\Gamma(k+1)(-u^\beta t_1^{\alpha_1}t_2^{\alpha_2})^k}{\Gamma(k\alpha_1+1)\Gamma(k\alpha_2+1)}={}_2\Psi_2\Bigg[\begin{matrix}
			(1,1),&(1,1)\\\\
			(1,\alpha_1),&(1,\alpha_2)
		\end{matrix}\bigg|-u^\beta t_1^{\alpha_1}t_2^{\alpha_2}\Bigg],
	\end{equation*}
	where ${}_2\Psi_2$ is the generalized Wright function defined in (\ref{genwrit}).
\end{remark}
\begin{theorem}
	The distribution of STFPRF is given by
	\begin{equation}\label{stfprfdist}
		p_{\alpha_1,\alpha_2}^\beta(n,t_1,t_2)=\frac{(-1)^n}{n!}{}_3\Psi_3\Bigg[\begin{matrix}
			(1,\beta),&(1,1),&(1,1)\\\\
			(1-n,\beta),&(1,\alpha_1),&(1,\alpha_2)
		\end{matrix}\bigg|-t_1^{\alpha_1}t_2^{\alpha_2}\lambda^\beta\Bigg],\ n\ge0.
	\end{equation}
\end{theorem}
\begin{proof}
	In view of (\ref{sfprfrep}) and by using (\ref{spt2dist}), the distribution of STFPRF is given by
	\begin{equation*}
		p_{\alpha_1,\alpha_2}^\beta(n,t_1,t_2)=\iint_{\mathbb{R}_+\times\mathbb{R}_+}\sum_{k=0}^{\infty}(-1)^{n+k}\frac{(\alpha k)_n(x_1x_2\lambda^\alpha)^k}{n!k!}\mathrm{Pr}\{T_{2\alpha_1}(t_1)\in\mathrm{d}x_1\}\mathrm{Pr}\{T_{2\alpha_2}(t_2)\in\mathrm{d}x_2\}.
	\end{equation*}
	Its double Laplace transform is
	\begin{align*}
	\iint_{\mathbb{R}_+\times\mathbb{R}_+}e^{-z_1t_1-z_2t_2}&p_{\alpha_1,\alpha_2}^\beta(n,t_1,t_2)\,\mathrm{d}t_1\,\mathrm{d}t_2\\
	&=\sum_{k=0}^{\infty}(-1)^{n+k}\frac{(\beta k)_n\lambda^{\beta k}}{n!k!}z_1^{\alpha_1-1}z_2^{\alpha_2-1}\int_{0}^{\infty}\int_{0}^{\infty}(x_1x_2)^ke^{-z_1^{\alpha_1}x_1}e^{-z_2^{\alpha_2}x_2}\,\mathrm{d}x_1\,\mathrm{d}x_2\\
	&=\sum_{k=0}^{\infty}(-1)^{n+k}\frac{(\beta k)_n\lambda^{\beta k}}{n!}\frac{k!z_1^{\alpha_1-1}z_2^{\alpha_2-1}}{z_1^{\alpha_1 k+\alpha_1}z_2^{\alpha_2 k+\alpha_2}},
	\end{align*}
	whose inverse Laplace transform with respect to $z_1$ and $z_2$ yields 
	\begin{equation*}
		p_{\alpha_1,\alpha_2}^\beta(n,t_1,t_2)=\frac{(-1)^n}{n!}\sum_{k=0}^{\infty}\frac{\Gamma(\beta k+1)\Gamma(k+1)\Gamma(k+1)}{\Gamma(\beta k+1-n)\Gamma(\alpha_1 k+1)\Gamma(\alpha_2 k+1)k!}(-t_1^{\alpha_1}t_2^{\alpha_2}\lambda^\beta)^k,\ n\ge0.
	\end{equation*}
	This completes the proof.
\end{proof}
\begin{remark}
	For $\beta=1$, (\ref{stfprfdist}) reduces to 
	\begin{align*}
		p_{\alpha_1,\alpha_2}(n,t_1,t_2)&=\frac{(-1)^n}{n!}\sum_{k=n}^{\infty}\frac{\Gamma( k+1)\Gamma(k+1)\Gamma(k+1)}{\Gamma(k+1-n)\Gamma(\alpha_1 k+1)\Gamma(\alpha_2 k+1)k!}(-t_1^{\alpha_1}t_2^{\alpha_2}\lambda)^k\\
		&=\sum_{k=n}^{\infty}\frac{(-1)^{n+k}(k)_n(k)_{k-n}}{\Gamma(\alpha_1 k+1)\Gamma(\alpha_2 k+1)}(t_1^{\alpha_1}t_2^{\alpha_2}\lambda)^k,\ n\ge0,
	\end{align*}
	which is the distribution of TFPRF (see \cite{Kataria2024}). 
	Also, it Laplace transform is
	\begin{align}
		\mathbb{E}e^{-uN_{\alpha_1,\alpha_2}(t_1,t_2)}&=\sum_{n=0}^{\infty}e^{-un}\sum_{k=n}^{\infty}\frac{(-1)^{n+k}(k)_n(k)_{k-n}}{\Gamma(\alpha_1 k+1)\Gamma(\alpha_2 k+1)}(t_1^{\alpha_1}t_2^{\alpha_2}\lambda)^k\nonumber\\
		&=\sum_{k=0}^{\infty}\frac{k!(-t_1^{\alpha_1}t_2^{\alpha_2}\lambda)^k}{\Gamma(\alpha_1 k+1)\Gamma(\alpha_2 k+1)}\sum_{n=0}^{k}\binom{k}{n}(-e^{-u})^n\nonumber\\
		&=\sum_{k=0}^{\infty}\frac{k!(-t_1^{\alpha_1}t_2^{\alpha_2}\lambda(1-e^{-u}))^k}{\Gamma(\alpha_1 k+1)\Gamma(\alpha_2 k+1)}\nonumber\\
		&={}_2\Psi_2\Bigg[\begin{matrix}
			(1,1),&(1,1)\\\\
			(1,\alpha_1),&(1,\alpha_2)
		\end{matrix}\bigg|-t_1^{\alpha_1}t_2^{\alpha_2}\lambda(1-e^{-u})\Bigg],\ u>0.\label{tfprflap}
	\end{align}
	Moreover, for $\alpha_1=\alpha_2=1$, (\ref{stfprfdist}) reduces to the distribution of SFPRF.
\end{remark}

Let $\{N_{\alpha_1,\alpha_2}(t_1,t_2),\ (t_1,t_2)\in\mathbb{R}^2_+\}$ be the TFPRF and $\{T_{2\alpha_i}(t_i),\ t_i\ge0\}$, $i=1,2$ be random process as defined in Theorem \ref{thm5}. Then,  $N_{\alpha_1,\alpha_2}(t_1,t_2)\overset{d}{=}N(T_{2\alpha_1}(t_1),T_{2\alpha_2}(t_2))$ (see \cite{Kataria2024}), where $\{N(t_1,t_2),\ (t_1,t_2)\in\mathbb{R}^2_+\}$ is the PRF. It is assumed that all these processes are independent of each other.

In the next result, we give an alternate time-changed representation of TFPRF using the one parameter Poisson process.
\begin{theorem}\label{thm54}
Let $\{N(t),\ t\ge0\}$ be a Poisson process with transition rate $\lambda>0$ and $\{T_{2\alpha_i}(t_i),\ t_i\ge0\}$, $i=1,2$ be random processes as defined in Theorem \ref{thm5}. Also, it is assumed that all these processes are independent of each other. Then, the following holds true for TFPRF: 
\begin{equation}\label{tfprfrep}
	N_{\alpha_1,\alpha_2}(t_1,t_2)\overset{d}{=}N(T_{2\alpha_1}(t_1)T_{2\alpha_2}(t_2)),\ (t_1,t_2)\in\mathbb{R}^2_+.
\end{equation}
\end{theorem} 
\begin{proof}
	The Laplace transform of the right hand side of (\ref{tfprfrep}) is given by
	\begin{align*}
		\mathbb{E}\exp(-uN(T_{2\alpha_1}(t_1)T_{2\alpha_2}(t_2)))&=\mathbb{E}(\mathbb{E}(\exp(-uN(T_{2\alpha_1}(t_1)T_{2\alpha_2}(t_2)))|T_{2\alpha_1}(t_1)T_{2\alpha_2}(t_2))\\
		&=\mathbb{E}\exp\big(\lambda T_{2\alpha_1}(t_1)T_{2\alpha_2}(t_2)(e^{-u}-1)\big),\ u>0.
	\end{align*}
	On following the proof in Remark \ref{rem5}, it can be shown that $\mathbb{E}\exp(-uN(T_{2\alpha_1}(t_1)T_{2\alpha_2}(t_2)))$ coincides with (\ref{tfprflap}). This completes the proof.
\end{proof}
\subsection{Space-time fractional CPRF with normal compounding} Let $Z_1,Z_2,\dots$ be iid standard normal random variables independent of the STFPRF $\{N_{\alpha_1,\alpha_2}^\beta(t_1,t_2),\ (t_1,t_2)\in\mathbb{R}^2_+\}$. Also, let 
\begin{equation}\label{stfcprf}
	\sum_{j=1}^{N_{\alpha_1,\alpha_2}^\beta(t_1,t_2)}Z_j,\ (t_1,t_2)\in\mathbb{R}^2_+,\ 0<\beta\leq1,\ 0<\alpha_i\leq1,\ i=1,2
\end{equation}
be the space-time fractional CPRF. Its Fourier transform
is given by
\begin{align*}
\mathbb{E}\exp&\bigg(-\frac{u^2N_{\alpha_1,\alpha_2}^\beta(t_1,t_2)}{2}\bigg)\\
&=\iint_{\mathbb{R}_+\times\mathbb{R}_+}\mathbb{E}\exp\bigg(-\frac{u^2N_\beta(x_1,x_2)}{2}\bigg)\mathrm{Pr}\{T_{2\alpha_1}(t_1)\in\mathrm{d}x_1\}\mathrm{Pr}\{T_{2\alpha_2}(t_2)\in\mathrm{d}x_2\},\ u\in\mathbb{R},
\end{align*}
where we have used the time-changed representation given in (\ref{sfprfrep}). Here, $\{N_\beta(t_1,t_2),\ (t_1,t_2)\in\mathbb{R}^2_+\}$ is the SFPRF and $\{T_{2\alpha_i}(t),\ t\ge0\}$, $i=1,2$ are processes as defined in Theorem \ref{thm5}.

On using (\ref{spt2pgf}), we get
\begin{align*}
\mathbb{E}\exp\bigg(-&\frac{u^2N_{\alpha_1,\alpha_2}^\beta(t_1,t_2)}{2}\bigg)\\
&=\iint_{\mathbb{R}_+\times\mathbb{R}_+}\exp\bigg(-x_1x_2\lambda^\beta(1-e^{-u^2/2})^\beta\bigg)\mathrm{Pr}\{L_{\alpha_1}(t_1)\in\mathrm{d}x_1\}\mathrm{Pr}\{L_{\alpha_2}(t_2)\in\mathrm{d}x_2\}\\
&={}_2\Psi_2\Bigg[\begin{matrix}
	(1,1),&(1,1)\\\\
	(1,\alpha_1),&(1,\alpha_2)
\end{matrix}\bigg|-\lambda^\beta(1-e^{-u^2/2})^\beta t_1^{\alpha_1}t_2^{\alpha_2}\Bigg].
\end{align*}
\subsubsection{Limiting result} Let $\{N_{\alpha_1,\alpha_2}^{\beta, (n^2)}(t_1,t_2),\ (t_1,t_2)\in\mathbb{R}^2_+\}$ be the STFPRF with parameter $n^2>0$. Then, the Fourier transform of the following scaled space-time fractional CPRF:
\begin{equation}\label{stfcprfs}
	n^{-1}\sum_{j=1}^{N_{\alpha_1,\alpha_2}^{\beta, (n^2)}(t_1,t_2)}Z_j,\ n>0,
\end{equation}
is given by
\begin{equation}\label{stfccprf}
	{}_2\Psi_2\Bigg[\begin{matrix}
		(1,1),&(1,1)\\\\
		(1,\alpha_1),&(1,\alpha_2)
	\end{matrix}\bigg|-n^{2\beta}(1-e^{-u^2n^{-2}/2})^\beta t_1^{\alpha_1}t_2^{\alpha_2}\Bigg],\ u\in\mathbb{R}.
\end{equation}
On letting $n\to\infty$, (\ref{stfccprf}) converges to
\begin{equation}\label{bscomp1lap}
	{}_2\Psi_2\Bigg[\begin{matrix}
		(1,1),&(1,1)\\\\
		(1,\alpha_1),&(1,\alpha_2)
	\end{matrix}\bigg|-(u^2/2)^\beta t_1^{\alpha_1}t_2^{\alpha_2}\Bigg],\ u\in\mathbb{R}.
\end{equation}

Let $\{B(t),\ t\ge0\}$ be the standard Brownian motion, and let $\{H_\beta(t_1,t_2),\ (t_1,t_2)\in\mathbb{R}^2_+\}$ be two parameter stable subordinator. Let us consider the following two parameter process:
\begin{equation}\label{tpsubcomp}
	B(H_\beta(T_{2\alpha_1}(t_1),T_{2\alpha_2}(t_2))),\ (t_1,t_2)\in\mathbb{R}^2_+.
\end{equation}
It is assumed that all the component processes in (\ref{tpsubcomp}) are independent of each other.

In view of Remark \ref{rem5}, the Fourier transform of (\ref{tpsubcomp}) is given by
\begin{align*}
	\mathbb{E}\exp(iuB(H_\beta(T_{2\alpha_1}(t_1),T_{2\alpha_2}(t_2))))&=\int_{0}^{\infty}e^{-u^2x/2}\mathrm{Pr}\{H_\beta(T_{2\alpha_1}(t_1),T_{2\alpha_2}(t_2))\in\mathrm{d}x\}\\
	&={}_2\Psi_2\Bigg[\begin{matrix}
		(1,1),&(1,1)\\\\
		(1,\alpha_1),&(1,\alpha_2)
	\end{matrix}\bigg|-(u^2/2)^\beta t_1^{\alpha_1}t_2^{\alpha_2}\Bigg],
\end{align*}
which coincides with (\ref{bscomp1lap}). Thus, the random field (\ref{stfcprfs}) converges weakly to (\ref{tpsubcomp}) for all $(t_1,t_2)\in\mathbb{R}^2_+$, as $n\to\infty$.
\begin{remark}
	For $\alpha_1=\alpha_2=1$, the space-time fractional CPRF (\ref{stfcprf}) reduces to the following space fractional CPRF:
	\begin{equation}\label{sfcprf}
		\sum_{j=1}^{N_\beta(t_1,t_2)}Z_j,\ (t_1,t_2)\in\mathbb{R}^2_+,\ 0<\beta\leq1,
	\end{equation}
	where $\{N_\beta(t_1,t_2),\ (t_1,t_2)\in\mathbb{R}^2_+\}$ is the space fractional PRF. Its Fourier transform is $\exp(-t_1t_2\lambda^\beta(1-u^2/2)^\beta)$, $u\in\mathbb{R}$. Also, if $\{N_\beta^{(n^2)}(t_1,t_2),\ (t_1,t_2)\in\mathbb{R}^2_+\}$ is a space fractional PRF with parameter $n^2>0$ then 
	\begin{equation*}
		\lim_{n\to\infty}n^{-1}\sum_{j=1}^{N_\beta^{(n^2)}(t_1,t_2)}Z_j\overset{d}{=}B(H_\beta(t_1,t_2)),\ (t_1,t_2)\in\mathbb{R}^2_+,
	\end{equation*}
	where the standard Brownian motion $\{B(t),\ t\ge0\}$ is independent of $\{H_\beta(t_1,t_2),\ (t_1,t_2)\in\mathbb{R}^2\}$.
\end{remark}
\begin{remark}
	From Theorem \ref{thm54}, it follows that
	\begin{equation}\label{c}
		\lim_{n\to\infty}n^{-1}\sum_{j=1}^{N_{\alpha_1,\alpha_2}^{\beta, (n^2)}(t_1,t_2)}Z_j\bigg|_{\beta=1}\overset{d}{=}B(T_{2\alpha_1}(t_1)T_{2\alpha_2}(t_2)),\ (t_1,t_2)\in\mathbb{R}^2_+,
	\end{equation}
	where all the processes appearing on the right hand side of (\ref{c}) are independent of each other.
\end{remark}
\section{Conclusion}
We introduced and studied some fractional compound Poisson random fields on $\mathbb{R}^2_+$. We obtained their distributions and the associated system of governing differential equations. In a particular case where the compounding random variables follow standard normal distribution, we examined the weak convergence of appropriately scaled versions of these fractional CPRFs. In this study, we are unable to establish the functional convergence of these processes. However, we intend to explore this aspect in future works.

\end{document}